\documentclass[A4paper]{article}
\usepackage[english]{babel}
\usepackage[T1]{fontenc}
\usepackage[utf8]{inputenc}
\usepackage{amsmath}
\usepackage{amsthm}
\usepackage{amsfonts}
\usepackage{mathrsfs}
\theoremstyle{definition}

\usepackage{xcolor}
\usepackage{amssymb}
\usepackage{graphicx}
\usepackage{float}
\usepackage{subfig}
\usepackage{stmaryrd}
\usepackage{caption}
\usepackage{physics}
\usepackage{cases}
\usepackage{listings}
\usepackage{color} 
\usepackage{booktabs,tabularx,caption}
\usepackage{multirow}

\usepackage{enumitem}

\usepackage{thmtools}

\usepackage{algorithm}
\usepackage{algpseudocode}

\usepackage{geometry}
\usepackage{microtype}

\usepackage[numbers]{natbib}

\definecolor{mygreen}{RGB}{28,172,0} 
\definecolor{mylilas}{RGB}{170,55,241}
\theoremstyle{definition}
\newtheorem{definition}{Definition}[section]
\newtheorem{proposition}{Proposition}
\newtheorem{theorem}{Theorem}
\newtheorem{lemma}{Lemma}

\newtheorem{remark}{Remark}

\declaretheoremstyle[
  bodyfont=\normalfont\itshape,
  headformat=\NAME\NUMBER  
]{nospacetheorem}
\declaretheorem[style=nospacetheorem,name=A]{assumption}

\declaretheoremstyle[
  bodyfont=\normalfont\itshape,
  headformat=\NAME\NUMBER  
]{nospacetheorem}
\declaretheorem[style=nospacetheorem,name=H]{hypothesis}

\usepackage{makecell}
\usepackage{hyperref}
\usepackage{cleveref}
\allowdisplaybreaks
\let\oldbibliography\thebibliography
\renewcommand{\thebibliography}[1]{\oldbibliography{#1}
\setlength{\itemsep}{0pt}} 

\newcommand{\vertiii}[1]{{\left\vert\kern-0.25ex\left\vert\kern-0.25ex\left\vert #1 
    \right\vert\kern-0.25ex\right\vert\kern-0.25ex\right\vert}}

\newcommand{\footremember}[2]{%
    \footnote{#2}
    \newcounter{#1}
    \setcounter{#1}{\value{footnote}}%
}
\newcommand{\footrecall}[1]{%
    \footnotemark[\value{#1}]%
} 

\usepackage{mwe}

\newlength{\tempdima}
\newcommand{\rowname}[1]
{\rotatebox{90}{\makebox[\tempdima][c]{\textbf{#1}}}}

\renewcommand{\thesubfigure}{\alph{subfigure}}
\newcommand{\mycaption}[1]
{\refstepcounter{subfigure}\textbf{(\thesubfigure) }{\ignorespaces #1}}

\begin{document}
\author{P.F. Antonietti\footremember{alley}{MOX, Department of Mathematics, Politecnico di Milano, Italy 
  (paola.antonietti@polimi.it, marco.verani@polimi.it)}, S. Berrone\footremember{trailer}{Department of Applied Mathematics, Politecnico di Torino, Italy
  (stefano.berrone@polito.it, martina.busetto@polito.it).}, M. Busetto\footrecall{trailer}{}, M. Verani\footrecall{alley}{}}
\title{Agglomeration-based geometric multigrid schemes for the Virtual Element Method}
\date{}
\maketitle

\begin{abstract}
In this paper we analyse the convergence properties of two-level, W-cycle and V-cycle agglomeration-based geometric multigrid schemes for the numerical solution of the linear system of equations stemming from the lowest order $C^{0}$-conforming Virtual Element discretization of two-dimensional second-order elliptic partial differential equations. The sequence of agglomerated tessellations are nested, but the corresponding multilevel virtual discrete spaces are generally non-nested thus resulting into non-nested multigrid algorithms. 
We prove the uniform convergence of the two-level method with respect to the mesh size and the uniform convergence of the W-cycle and the V-cycle multigrid algorithms with respect to the mesh size and the number of levels.   Numerical experiments confirm the theoretical findings.
\end{abstract}

\vspace{0.5cm}

\textbf{Keywords}: geometric multigrid algorithms, agglomeration, virtual element method, elliptic problems, polygonal meshes

\vspace{0.5cm}

\textbf{AMS}: 65N22, 65N30, 65N55

\section{Introduction}
The Virtual Element Method (VEM) is a very recent extension of the Finite Element Method (FEM) originally introduced in \cite{beirao2013basic} for the discretization of the Poisson problem on fairly general polytopal meshes. From its original introduction, the VEM has been applied to a variety of problems  \cite{Antonietti2022, beirao2022}.
However, the design of efficient solvers for the solution of the linear system stemming from the virtual element discretization is still a relatively unexplored field of research. So far, the few existing works in literature have mainly focused on the study of the condition number of the stiffness matrix due to either the increase in the order of the method or to the degradation of the quality of the meshes \cite{berrone2017orthogonal, mascotto2018ill} and on the development of preconditioners based on domain decomposition techniques \cite{calvo2019overlapping, calvo2018approximation, bertoluzza2017bddc, prada2017feti, dassi2020parallel}. Instead, the analysis of multigrid methods for VEM is much less developed. In particular, \cite{antonietti2018multigrid} presents the development of an efficient geometric multigrid (GMG) algorithm for the iterative solution of the linear system of equations stemming from the p-version of the Virtual Element discretization of two-dimensional Poisson problems, whereas \cite{prada2018algebraic} presents the development of an efficient algebraic multigrid (AMG) method for the solution of the system of equations related to the Virtual Element discretization of elliptic problems. To the best of our knowledge, the design and analysis of a GMG method for the h-version of the VEM  has not been investigated yet.

In this paper, hinging upon the geometric flexibility of VEM, we consider agglomerated grids and focus on the analysis of geometric multigrid methods (two-level, W-cycle, V-cycle) for the h-version of the lowest order virtual element method. It is worth noticing that the idea of exploiting the flexibility of the element shape  has been investigated in  \cite{antonietti2017multigrid, antonietti2019v}  where multigrid methods for the numerical solution of the linear system of equations stemming from the discontinuous Galerkin discretization of second-order elliptic partial differential equations have been analysed. 

Throughout this paper we mainly consider nested sequences of agglomerated meshes obtained from a fine grid of triangles by applying a recursive coarsening strategy. 
It is crucial to underline that even if the tessellations are nested, the corresponding multilevel discrete virtual element spaces are not. Therefore, our approach results into a non-nested multigrid method. A generalized framework for non-nested multilevel methods was developed by Bramble, Pasciak and Xu in \cite{bramble1991analysis} and later extended by Duan, Gao, Tan and Zhang in \cite{duan2007generalized} to analyze the non-nested V-cycle methods.  Following the so called BPX framework, we study the convergence of our method. In particular,  we prove that, under suitable assumptions on the quality of the agglomerated coarse grids, our two-level iterative method converges uniformly with respect to the granularity of the mesh for inherited and non-inherited bilinear forms. Moreover, we prove that  the W-cycle and V-cycle schemes with non-nested virtual element spaces converge uniformly with respect to the mesh size and the number of levels for both inherited and non-inherited bilinear forms. In the case of non-inherited bilinear form the W-cycle and V-cycle schemes are proved to converge provided that a sufficiently large number of smoothing steps is chosen. The theoretical results are confirmed by the numerical experiments.

The outline of the paper is as follows. In \Cref{Model_problem}  we describe the model problem and its Virtual Element discretization. In \Cref{Multigrid_algorithms} we introduce the two-level, the W-cycle and V-cycle multigrid virtual element methods. In  \Cref{Coarsening_strategy} we present the coarsening strategy adopted to construct the sequence of nested meshes, while in \Cref{Prolongation_operator} we define suitable prolongation operators that are a key ingredient in multilevel methods. In  \Cref{BPX_framework} we introduce the BPX framework for the theoretical convergence analysis of our multigrid schemes, while in \Cref{Convergence_analysis} we analyse the convergence of our virtual element multigrid algorithm and state the main theoretical results. 
In \Cref{Numerical_experiments} we present the algebraic counterpart of the algorithm focusing on the implementation details and we discuss some numerical results obtained applying the method to the numerical solution of the linear systems stemming from the h-version of the lowest order Virtual Element discretization of the Poisson equation. Finally, in \Cref{Conclusions} we draw some conclusions.

Throughout this paper, we use the notation $x \lesssim y$ and $x \gtrsim y$ instead of $x \leq C y$ and $x \geq C y$, respectively, where $C$ is a positive constant independent of the mesh size. When needed the constant will be written explicitly. Moreover, $\mathbb{P}_{l}(D)$ denotes the space of polynomials of degree less than or equal to $l \ge 1$ on the open bounded domain $D$ and $[\mathbb{P}_{l}(D)]^{2}$ the corresponding vector-valued space. 

\section{Model problem} \label{Model_problem}
Let $\Omega \subset \mathbb{R}^{2}$ be a convex polygonal domain  with Lipschitz boundary and let $f \in L^{2}(\Omega)$. We consider the following model problem: find $u \in V := H^{1}_{0}(\Omega)$ such that
\begin{equation}\label{elliptic_problem}
\mathcal{A}(u,v) = (f, v)_{L^{2}(\Omega)} \ \ \ \forall v \in V,
\end{equation} 
where $\mathcal{A}(u, v) := (\mu \nabla u, \nabla v)_{L^{2}(\Omega)}$ with $ \mu \in L^{\infty}(\Omega)$ a positive constant. This problem is well-posed and its unique solution $u \in H^{2}(\Omega)$ satisfies
\begin{equation}\label{elliptic_regularity}
\|u\|_{H^{2}(\Omega)} \lesssim \|f\|_{L^{2}(\Omega)}.
\end{equation}
For the analysis under weaker regularity assumptions see, e.g. \cite{brenner1999convergence}.

For the purposes of this work, we consider a sequence $\{\mathcal{T}_{j}\}_{j = 1}^{J}$ of tessellations of the domain $\Omega$. Therefore, all the parameters characterizing a given tessellation $\mathcal{T}_{j}$ will be denoted by the subscript $j$.
Each tessellation is made of disjoint open polytopic elements $E_{j}$ such that $\bar{\Omega} := \bigcup_{E_{j} \in \mathcal{T}_{j}} \bar{E}_{j}, \ j = 1, \dots, J$. For each element $E_{j}$, we denote by $\mathcal{E}_{E_{j}}$ the set of its edges and by $\delta_{E_j}$ its diameter. The mesh size of $\mathcal{T}_{j}$ is denoted by $\delta_j := \max_{E_{j} \in \mathcal{T}_{j}} \delta_{E_{j}}$. We assume that the elements $E_{j}$ of each tessellation $\mathcal{T}_{j}$ satisfy the following assumptions \cite{beirao2016virtual}.

\begin{assumption}\label{mesh_assumption}
For any $j = 1, \dots, J$, every element $E_{j} \in \mathcal{T}_{j}$ is the union of a finite and uniformly bounded number of star-shaped domains with respect to a disk of radius $\rho_{E_{j}} \delta_{E_{j}}$ and every edge $e_j \in \mathcal{E}_{E_{j}}$ must be such that $|e_j| \geq \rho_{E_{j}} \delta_{E_{j}}$, being $|e_{j}|$ its length. Moreover, given a sequence of tessellations $\{ \mathcal{T}_{j} \}_{j=1}^{J}$ there exists a $\rho_{0}$ independent of the tessellation such that $\rho_{E_{j}} \geq \rho_{0} > 0$.
\end{assumption}

\begin{assumption}\label{quasi_uniform_mesh}
\textit{The sequence of tessellations $\{\mathcal{T}_{j}\}_{j=0}^{J}$ are quasi-uniform, i.e., they are regular and there exists a constant $\tau > 0$ such that}
\begin{equation*}
\min_{E_{j} \in \mathcal{T}_{j}} \delta_{E_{j}} \geq \tau \delta_{j} \ \ \ \forall \delta_{j} > 0.
\end{equation*}
\textit{Moreover,  $\{\mathcal{T}_{j}\}_{j=0}^{J}$ satisfies a bounded variation hypothesis between subsequent levels, i.e., $\delta_{j-1} \lesssim \delta_{j} \leq \delta_{j-1} \  \forall j = 2, \dots, J.$}
\end{assumption}
We introduce the h-version of the enhanced Virtual Element Method and we associate to each $\mathcal{T}_{j}$ the corresponding global virtual element space $V_{j}$ of order $k=1$, constructed from the local element spaces ${V}^{E_j}$ defined on each element $E_j \in \mathcal{T}_{j}$.

We define
\begin{equation*}
\mathbb{B}_{1}(\partial E_j) := \Big \{ v \in C^{0}(\partial E_j): v_{|_{e_j}} \in \mathbb{P}_{1}(e_j) \ \ \ \forall e_j \in \mathcal{E}_{E_j} \Big \},
\end{equation*} 
and the local enhanced virtual element space ${V}^{E_j}$ of order $k=1$ as
\begin{align*}
{V}^{E_j} := \Big \{ & v \in H^{1}(E_j): v_{|_{\partial E_j}} \in \mathbb{B}_{1}(\partial E_j), \ \Delta v_{|_{E_j}} \in \mathbb{P}_{1}(E_j), \\& (v,p)_{L^{2}(E_j)} = (\Pi^{\nabla}_{1,E} v , p)_{L^{2}(E_j)} \ \ \ \forall p \in \mathbb{P}_{1}(E_j)  \Big \}.
\end{align*}
Here, $\Pi^{\nabla}_{1,E_j}: H^{1}(E_j) \to \mathbb{P}_{1}(E_j)$ is the $H^{1}(E_j)$-orthogonal operator, defined as
\begin{align*}
(\nabla \Pi^{\nabla}_{1,E_j} v, \nabla p)_{L^{2}(E_j)} &= (\nabla v, \nabla p)_{L^{2}(E_j)}   \ \ \ \ \ \forall p \in \mathbb{P}_{1}(E_j), \\
(\Pi^{\nabla}_{1,E} v, 1)_{L^{2}(\partial  E_j)} &= (v,1)_{ L^{2}(\partial E_j)}.
\end{align*}
As a basis for the local polynomial space $\mathbb{P}_1(E_j)$, we choose the set of scaled monomials defined as
\begin{equation}\label{scaled_monomials}
\mathcal{M}_{1}(E_j) := \Big \{ m \in \mathbb{P}_{1}(E_j) : m(x,y) := \frac{(x - x_{E_j})^{{\alpha}_{x}}(y - y_{E_j})^{{\alpha}_{y}}}{\delta_{E_j}^{\alpha_x + \alpha_y}}, \ 0 \leq \alpha_x + \alpha_y \leq 1 \Big \},
\end{equation}
where $(x_{E_j},y_{E_j})$ are the coordinates of the center of the disk in respect of which the element $E_{j}$ is star-shaped. We denote by $N_{1} = 3$ the dimension of the local polynomial space $\mathbb{P}_{1}(E_{j})$.

As set of degrees of freedom of the local virtual element space ${V}^{E_j}$, we choose the standard set consisting of the values of $v \in {V}^{E_j}$ at the vertices of the polygon $E_j$. We denote by $N_{dof}^{E_j}$ the total number of degrees of freedom of ${V}^{E_j}$ and by  $\mathcal{N}(E_{j})$ the set of the indices of the nodes relative to the element $E_{j} \in \mathcal{T}_{j}$. Therefore, $N_{dof}^{E_j} := \# \mathcal{N}(E_{j})$. Moreover, we denote by 
\begin{equation}\label{dof_i}
\text{dof}_{i}(v) := v(x_i) \ \ \ \forall i \in \mathcal{N}(E_{j}),
\end{equation}
the operator  returning the $i$-th degree of freedom of $v \in {V}^{E_j}$.

As basis functions for ${V}^{E_j}$, we choose the Lagrangian shape functions with respect to the degrees of freedom of the element $E_j$ , i.e., the $\varphi_{i}^{E_j}, i \in \mathcal{N}(E_{j})$ such that
$\varphi_i^{E_j}(x_l) = \delta_{li} \ \ \ \forall i, l \in \mathcal{N}(E_{j}).$
Consequently, $v \in {V}^{E_{j}}$ can be written with respect to the local VEM basis as
\begin{equation*}
v =  \sum_{i \in \mathcal{N}(E_j)} \text{dof}_{i}(v) \varphi_{i}^{E_j} =  \sum_{i \in \mathcal{N}(E_j)} v(x_i) \varphi_{i}^{E_j}.
\end{equation*}

In addition, we consider the $L^{2}(E_j)$-projection $\Pi^{0}_{1,E_j}: {V}^{E_j} \to \mathbb{P}_{1}(E_j)$ defined as
\begin{equation*}
(\Pi^{0}_{1,E_j}v, p)_{L^{2}(E_j)} = (v,p)_{L^{2}(E_j)} \ \ \ \forall p \in \mathbb{P}_{1}(E_j), 
\end{equation*}
and the projections of the derivatives $\Pi^{0}_{0,E_j} \frac{\partial}{\partial x}, \Pi^{0}_{0,E_j} \frac{\partial}{\partial y}: {V}^{E_j} \to \mathbb{P}_{0}(E_j)$ such that
\begin{align*}
\Big ( \Pi^{0}_{0,E_j} \frac{\partial v}{\partial x}, p \Big )_{L^{2}(E_j)} &= \Big ( \frac{\partial v}{\partial x}, p \Big )_{L^{2}(E_j)} \ \ \ \forall p \in \mathbb{P}_{0}(E_j), \ \forall v \in {V}^{E_j},\\
\Big ( \Pi^{0}_{0,E_j} \frac{\partial v}{\partial y}, p \Big )_{L^{2}(E_j)} &= \Big ( \frac{\partial v}{\partial y}, p \Big )_{L^{2}(E_j)} \ \ \ \forall p \in \mathbb{P}_{0}(E_j),\  \forall v \in {V}^{E_j}.
\end{align*}
We denote by $\Pi^{0}_{0,E_j} \nabla v$ the vector having $\Pi^{0}_{0,E_j} \frac{\partial v}{\partial x}$ and $\Pi^{0}_{0, E_j} \frac{\partial v}{\partial y}$ as components.

We recall the following result reported in \cite{beirao2016virtual}.

\begin{lemma}\label{lemma_projector}
For all $E_{j} \in \mathcal{T}_{j}$ and all smooth enough functions $u$ on $E_{j}$, it holds
\begin{equation}\label{lemma_projector_equation}
\| u - \Pi^{0}_{1, E_{j}} u \|_{L^{2}(E_{j})} \lesssim \delta_{j}^{s} |u|_{H^{s}(E_j)} \ \ \ s \in \mathbb{N}, \ \ \ s = \{1,2\},
\end{equation}
where the hidden constant depends on $\rho_{0}$ defined as in Assumption A\ref{mesh_assumption}.
\end{lemma}

The global virtual element space $V_{j}$ is defined as 
\begin{equation}\label{global_VEM}
V_{j} := \{ v \in H^{1}_{0}(\Omega): {v}_{|_{E_j}} \in {V}^{E_j} \ \ \ \forall E_j \in \mathcal{T}_{j} \}, \ \ \ j = 1, \dots, J.
\end{equation}

Its set of degrees of freedom can be defined similarly as done for the local space. We denote by $N_{dof}^{j}$ the total number of degrees of freedom of $V_{j}$ and by $\mathcal{N}(\mathcal{T}_{j}): = \cup_{E_{j} \in \mathcal{T}_{j}} \mathcal{N}(E_{j})$ the set of the indices of all the nodes of all the elements $E_{j}$ of the tessellation $ \mathcal{T}_{j}$ (excluding the nodes on the boundary of the domain $\partial \Omega$). Therefore, $N_{dof}^{j} := \# \mathcal{N}(\mathcal{T}_{j})$.

Similarly to the local space, we choose the Lagrangian set $\varphi_i^{j}, \ i \in \mathcal{N}(\mathcal{T}_{j})$ with respect to the global degrees of freedom as basis functions of ${V}_j$.  Consequently, $v \in V_{j}$ can be written with respect to the global VEM basis functions as
\begin{equation*}
v =  \sum_{i \in \mathcal{N}(\mathcal{T}_{j})} \text{dof}_{i}(v) \varphi_{i}^{j} =  \sum_{i \in \mathcal{N}(\mathcal{T}_{j})} v(x_i) \varphi_{i}^{j}.
\end{equation*}
We point out that $\varphi^{j}_{i|_{E_{j}}} = \varphi_{i}^{E_{j}}$, with $\varphi_{i}^{E_{j}}$ defined as above.

The VEM for the approximate solution of our model problem on the finest level grid $J$ is: find $u_{J} \in V_{J}$ such that 
\begin{equation}\label{PoissonVEM}
\mathcal{A}_{J}(u_{J}, v_{J}) = \langle f, v_{J} \rangle \ \ \ \forall v_{J} \in V_{J}.
\end{equation}
The bilinear form $\mathcal{A}_{J}(\cdot, \cdot)$ in \eqref{PoissonVEM} is defined as
\begin{equation}\label{form_A}
\begin{aligned}
\mathcal{A}_{J}(u_{J},v_{J}) &:= \sum_{E_J \in \mathcal{T}_J} \mathcal{A}^{E_J}_{J}(u_{J}, v_J) := \sum_{E_J \in \mathcal{T}_J} \Big [ (\mu \Pi^{0}_{0,E_J} \nabla u_J, \Pi^{0}_{0,E_J} \nabla v_J)_{L^{2}(E_J)}  \\
&+  \| \mu\|_{L^{\infty}(E_J)} S^{E_J} \Big ((I - \Pi^{\nabla}_{1,E_J} ) u_J,  (I - \Pi^{\nabla}_{1,E_J}) v_J \Big ) \Big ], 
\end{aligned}
\end{equation}
and the right-hand side $\langle f, v_{J} \rangle$ is defined as 
\begin{equation}\label{right_hand_side}
\langle f, v_{J} \rangle := \sum_{E_J \in \mathcal{T}_J} (f, \Pi^{0}_{0, E_{J}} v_{J})_{L^{2}(E_{J})}.
\end{equation}

For the stabilization form $S^{E_J}$ in \eqref{form_A} we consider the scalar product of the vectors of degrees of freedom of the two functions
\begin{equation*}
\begin{aligned}
& S^{E_J}((I - \Pi^{\nabla}_{1,E_J}) u_J, (I - \Pi^{\nabla}_{1,E_J}) v_J) 
 :=  \sum_{i \in \mathcal{N}(E_{J})} \text{dof}_i \Big ((I - \Pi^{\nabla}_{1,E_J}) u_J \Big ) \  \text{dof}_i \Big ((I - \Pi^{\nabla}_{1,E_J}) v_J \Big),
\end{aligned}
\end{equation*}
where $\text{dof}_{i}(\cdot)$ is defined as in \cref{dof_i}.

\section{Multigrid algorithms}\label{Multigrid_algorithms}
In this section we introduce the h-multigrid two-level, W-cycle and V-cycle  schemes to solve the VEM discrete formulation \eqref{PoissonVEM}.

Let $V_{j}, \ j = 1, \dots, J$, be the sequence of finite-dimensional virtual element spaces defined in \eqref{global_VEM}. In order to define the multigrid cycle, we introduce the following intergrid transfer operators. The prolongation operator (see \Cref{Prolongation_operator}) connecting the coarser space $V_{j-1}$ to the finer space $V_{j}, \ j = 2, \dots, J$, is denoted by $I^{j}_{j-1} : V_{j-1} \to V_{j}$, whereas the restriction operator $I^{j-1}_{j} : V_{j} \to V_{j-1}$ connecting the finer space $V_{j}$ to the coarser space $V_{j-1}, \ j = 2, \dots, J$, is defined as the adjoint of $I^{j}_{j-1}$ with respect to the inner product $(\cdot, \cdot)_{j}$, i.e.,
\begin{equation*}
(I^{j-1}_{j} w_{j}, v_{j-1})_{j-1} = (w_{j}, I^{j}_{j-1} v_{j-1})_{j}  \ \ \ \forall v_{j-1} \in V_{j-1},
\end{equation*}
where $(\cdot, \cdot)_{j}$ is the $L^{2}$ scalar product on $V_{j}, \ j = 1, \dots, J$.

Let $\mathcal{A}_{J}(\cdot, \cdot)$ be the symmetric positive definite discrete bilinear form defined as in \eqref{form_A}. On each level $j-1$, with $j = 2, \dots, J$, we define the symmetric and positive definite bilinear form $\mathcal{A}_{j-1}(\cdot, \cdot) : V_{j-1} \times V_{j-1} \to \mathbb{R}$ as follows.
\begin{definition}[Inherited and non-inherited bilinear forms]\label{inherited_non_inherited_bilinear_form}
The inherited bilinear form $\mathcal{A}_{j-1}(\cdot, \cdot)$ is defined as
\begin{equation*}
\mathcal{A}_{j-1}(u,v) := \mathcal{A}_{j}(I_{j-1}^{j} u, I_{j-1}^{j} v) \ \ \ \forall  u, v \in  V_{j-1}, \   \  j = 2, \dots, J.
\end{equation*}
The non-inherited bilinear form $\mathcal{A}_{j-1}(\cdot, \cdot)$ is defined as in \eqref{form_A}  but on the level $j-1$.
\end{definition}


We also introduce the operators $A_{j} : V_{j} \to V_{j},$ defined as
\begin{equation}\label{operator_A}
(A_j w, v)_{j} = \mathcal{A}_j(w,v) \ \ \ \forall w , v \in V_j, \ \ \ j = 1, \dots, J.
\end{equation}

For the theoretical analysis, we also need the operator $P_{j}^{j-1} : V_{j} \to V_{j-1}$ for $j = 2, \dots, J$, defined as
\begin{equation*}
\mathcal{A}_{j-1} (P^{j-1}_{j} w_{j}, v_{j-1}) = \mathcal{A}_j (w_{j}, I^{j}_{j-1} v_{j-1}) \ \ \ \forall v_{j-1} \in V_{j-1}, \ w_{j} \in V_{j}.
\end{equation*}

As a smoothing scheme, we choose the symmetric Gauss-Seidel method. However, we point out that other smoothing schemes can be selected. 
We denote by $R_{j} : V_{j} \to V_{j}$ the linear smoothing operator and by $R^{T}_{j}$ the adjoint operator of $R_{j}$ with respect to the selected inner product $(\cdot, \cdot)_{j}$. We set $R^{(l)}_{j}$ equals to $R_{j}$  if $l$ is odd and $R^{T}_{j}$ if $l$ is even.
                 
Now, we are ready to introduce the multigrid method \cite{bramble1991analysis}. We denote by $\nu$ the number of smoothing steps. Then, at the level $j$ with $j = 1, \dots, J$, the multigrid operator $B_{j}: V_{j} \to V_{j}$ is defined by induction in the following way. We set $B_1 := A_1^{-1}$ and given an initial iterate $x^{0}$, we define $B_{j} g \in V_{j}$ for $g \in V_{j}$ as in \cref{multigrid_algorithm}.

\begin{algorithm}
\footnotesize
\caption{Multigrid algorithm (MG) \ \ \ $B_{j} g = \text{MG}(p, j,g,x^{0},\nu)$}\label{multigrid_algorithm}
\begin{enumerate}
\item Set $q^{0} = 0$.
\item Define $x^{l}$ for $l = 1, \dots, \nu$ by
\begin{equation*}
x^{l} = x^{l-1} + R_{j}^{(l+\nu)}(g-A_{j}x^{l-1}).
\end{equation*}
\item Set $r_{j-1} = I^{j-1}_{j}(g - A_{j} x^{\nu})$
\item Define $q^{i}$ for $i = 1, \dots, p$ by
\begin{equation*}
q^{i} = \text{MG}(p,j-1,r_{j-1},q^{i-1}, \nu)
\end{equation*}
\item Set $y^{\nu} = x^{\nu} + I_{j-1}^{j} q^{p}$
\item Define $y^{l}$ for $l = \nu + 1, \dots, 2 \nu$ by
\begin{equation*}
y^{l} = y^{l-1} + R_{j}^{(l+\nu)}(g -A_{j} y^{l-1}).
\end{equation*}
\item Set $B_{j} g = y^{2 \nu}$.
\end{enumerate}
\end{algorithm}

The quantity $p$ is assumed to be a positive integer. We focus on the cases $p=1$ and $p=2$ that correspond to the symmetric $V$-cycle and the symmetric $W$-cycle, respectively. We underline that in Step 2 of the algorithm, we alternate between $R_{j}$ and $R_{j}^{T}$, whereas in Step 4, we use their adjoints applied in the reverse order. 

Furthermore, we introduce the following notation that will be useful in the convergence analysis.  We set ${K}_{j} := I - R_{j} A_{j}$, where $I$ is the identity operator, and we define its adjoint with respect to $\mathcal{A}_j(\cdot,\cdot)$ as ${K}_{j}^{*} := I - R^{T}_{j}A_{j}$. Moreover, we set
\begin{equation*}
\tilde{K}_{j}^{(\nu)} := 
\begin{cases}
(K^{*}_{j} K_{j})^{\frac{\nu}{2}} \ \ \ \ \ \ \ \ \ \text{if  $\nu$ is even}, \\
(K^{*}_{j} K_{j})^{\frac{\nu -1}{2}} K^{*}_{j} \ \ \text{if $\nu$ is odd}.
\end{cases}
\end{equation*} 

It can be proved (see \cite{bramble1987new}) that the following fundamental recursive relation for the multigrid operators $B_{j}$ introduced above holds true for $\ j = 2, \dots, J$
\begin{equation}\label{error_propagator_operator}
I - B_{j} A_{j} = (\tilde{K}_{j}^{(\nu)})^{*} [(I - I_{j-1}^{j} P^{j-1}_{j}) + I_{j-1}^{j}(I - B_{j-1} A_{j-1})^{p} P_{j}^{j-1}] \tilde{K}_{j}^{(\nu)}.
\end{equation}
The quantity $I - B_{j} A_{j}$ is known as the error propagation operator.

\section{Coarsening strategy} \label{Coarsening_strategy}
In this section, we describe the construction of the sequence of tessellations $\{\mathcal{T}_{j}\}_{j=1}^{J}$ by means of an agglomeration strategy.
Given the open bounded connected domain $\Omega \subset \mathbb{R}^{2}$, we introduce a tessellation $\mathcal{T}_{J}$ of triangular elements $E_{J}$ having characteristic mesh size $\delta_{J}$. Starting from this tessellation $\mathcal{T}_{J}$, by agglomeration we generate a sequence of coarser nested meshes $\{ \mathcal{T}_{j} \}_{j = 1}^{J}$, where $j$ refers to the level of the agglomeration process. For instance, $j=J-1$, denotes the mesh at level $J-1$, i.e. the mesh $\mathcal{T}_{J-1}$ generated by the agglomeration of the mesh $\mathcal{T}_{J}$. Examples of coarsening strategy are reported in Figure \ref{Meshes}, each column is obtained by the coarsening strategy.


The elements of each mesh $\mathcal{T}_{j}, \ j = 1, \dots,J$ can be expressed as the union of the triangular elements of the original fine mesh $\mathcal{T}_{J}$. More formally, each mesh $\mathcal{T}_{j}$ satisfies the following requirements.
\begin{enumerate}
\item $\mathcal{T}_{j-1}$ represents a disjoint partition of $\Omega$ into elements obtained by a suitable cluster of elements of the mesh $\mathcal{T}_{j}$.
\item Each element $E_{j-1} \in \mathcal{T}_{j-1}$ is an open bounded connected subset of the domain $\Omega$ and it is possible to find  a set $\mathcal{T}_{E_{j-1}} \subset \mathcal{T}_{j}$ such that $\bar{E}_{j-1} =  \bigcup_{E_j \in \mathcal{T}_{E_{j-1}}}\bar{E}_{j}.$
\item For every open polytopic element $E_{j} \in \mathcal{T}_{j}$ there exists $\mathcal{T}_{E_{j}}^{J} \subset \mathcal{T}_{J}$ such that $\bar{E}_{j} = \bigcup_{E_J \in \mathcal{T}_{E_{j}}^{J}} \bar{E}_{J}.$
\end{enumerate}

\begin{remark}
Given a fine-level tessellation $\mathcal{T}_{J}$ consisting of uniformly star-shaped triangular elements, a finite number of agglomeration steps  will produce a sequence of tessellations such that every element $E_{j} \in \mathcal{T}_{j}, \ j = 1, \dots, J-1$, satisfies the above requirements and it is the union of a finite and uniformly bounded number of star-shaped domains with respect to a disk of radius $\rho_{E_{j}} \delta_{E_{j}}$ as required by Assumption A\ref{mesh_assumption}. In particular, we can select the $\rho_{0}$ in Assumption A\ref{mesh_assumption} to be the infimum of the values achieved by $\rho_{E_{j}}$ on all the considered tessellations $\mathcal{T}_{j}$.

As explained, the coarse tessellation $\mathcal{T}_{j-1}$ is obtained by agglomeration of the fine tessellation $\mathcal{T}_{j}$ and, in practice, each $E_{j-1}$ will be given by the bounded union of elements $E_{j} \in \mathcal{T}_{j}$. Consequently, in practical applications, the bounded variation hypothesis $\delta_{j-1} \lesssim \delta_{j} \leq \delta_{j-1}, \  j = 2, \dots, J$ in Assumption A\ref{quasi_uniform_mesh} is usually satisfied by construction.
\end{remark}

\begin{remark}
In general, what follows applies also to other nested meshes satisfying the following boundary compatibility condition, i.e, the edges of the element $E_{j} \in \tilde{E}_{j}^{j-1}$ that lie on the boundary of the element $E_{j-1}$ share the same nodes of the element $E_{j-1}$.
In Figure \ref{compatible_non_compatible}, we report an example of nested elements that satisfy and that do not satisfy the boundary compatibility condition. 
\end{remark}

Since the coarse level $\mathcal{T}_{j-1}, \ j  = 2, \dots, J$, is obtained by agglomeration from $\mathcal{T}_{j}$, the  partitions $\{\mathcal{T}_{j}\}_{j=1}^{J}$ are nested and this is of fundamental importance for the theoretical analysis that we will perform. We underline that even if the partitions satisfies a nestedness property, in general  the finite-dimensional spaces $\{V_{j}\}_{j=1}^{J}$ are non-nested. Indeed, $V_{j-1} \not\subset V_j, \  j = 2, \dots, J$. Consequently, the analysis of the proposed method will make use of the general framework of non-nested multigrid methods.

\begin{figure}
\centering
\captionsetup{justification=centering}
\subfloat[][Admissible nested elements.]{
\includegraphics[width=0.35\textwidth]{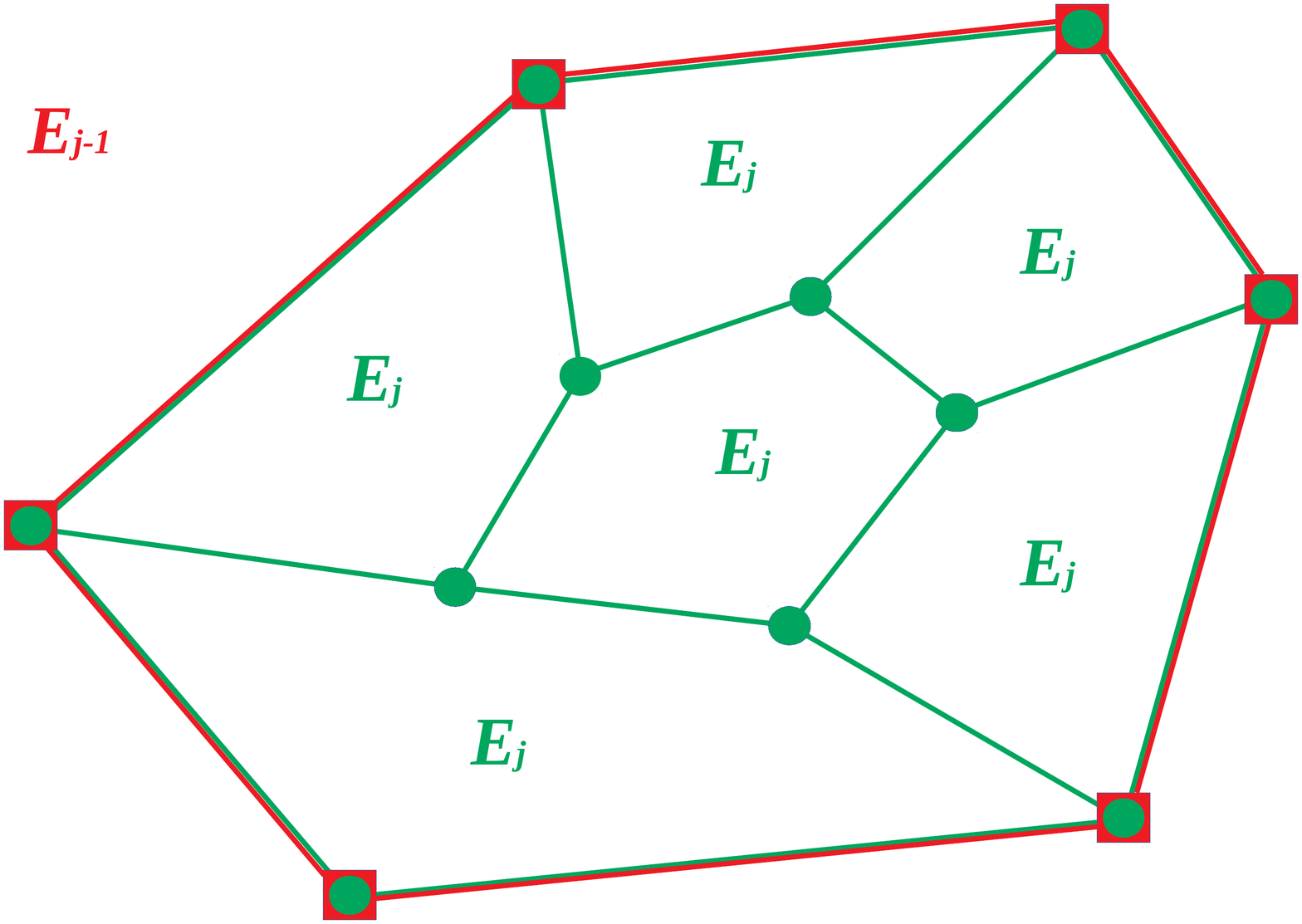}}
\subfloat[][Non admissible nested elements.]{
\includegraphics[width=0.35\textwidth]{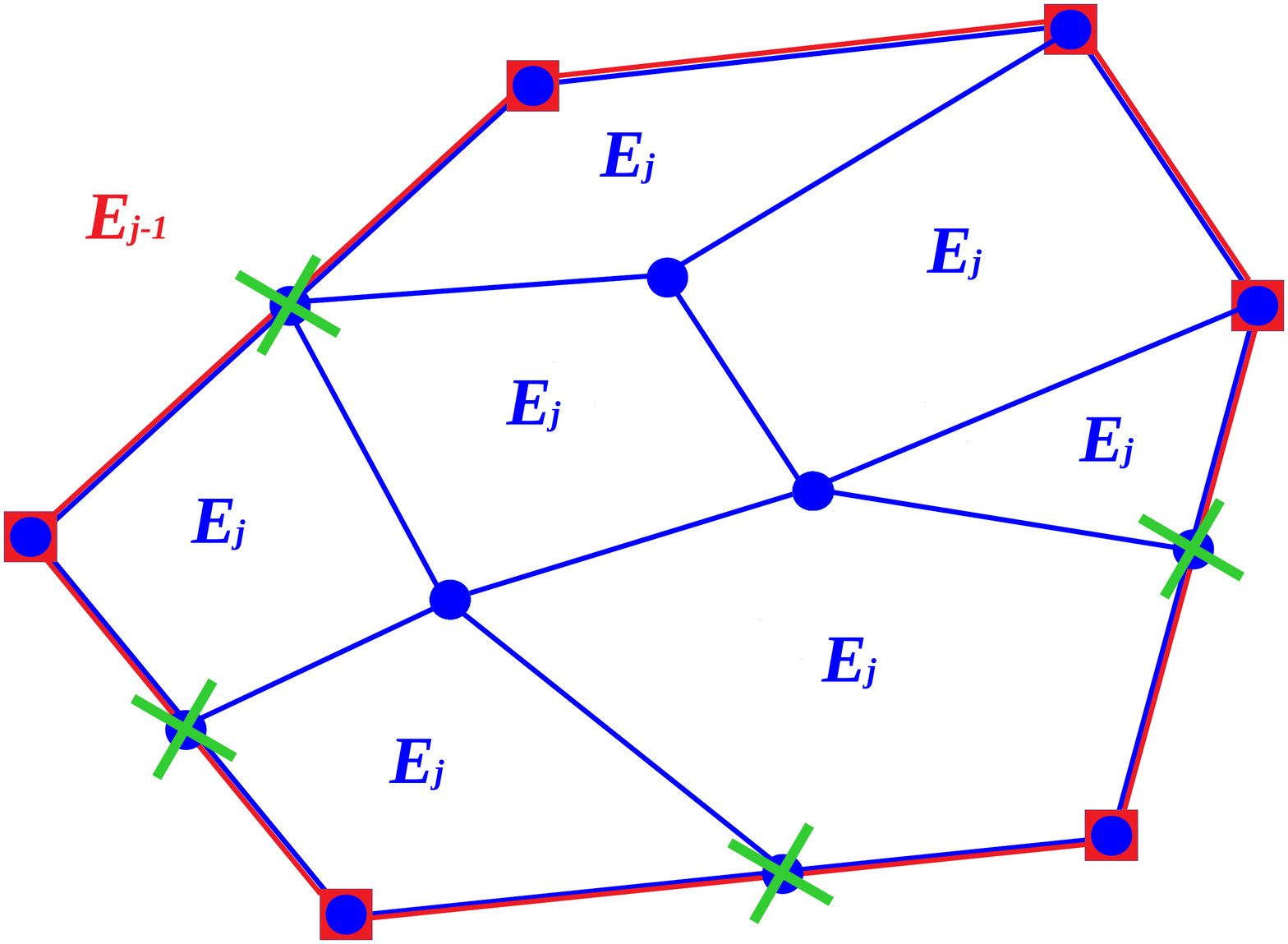}}
\caption{Example of (a) admissible nested elements   and (b) non-admissible nested elements. Circles and squares represent the nodes of $E_j$ and $E_{j-1}$, respectively. The green cross markers in (b) highlight nodes violating the boundary compatibility condition.}
\label{compatible_non_compatible}
\end{figure}

\section{Prolongation operator}\label{Prolongation_operator}
We underline that since in general $V_{j-1} \not\subset V_j$, the prolongation operator $I^{j}_{j-1}$ cannot be chosen as the classical injection operator.
In order to define the prolongation operator $I^{j}_{j-1}: V_{j-1} \to V_{j}$, we introduce the following notation. 

Let $\mathcal{T}_{E_{j-1}}$ the subset of $\mathcal{T}_{j}$  made of elements $E_{j} \in \mathcal{T}_{j}$ introduced in \Cref{Coarsening_strategy}, i.e.,  $\mathcal{T}_{E_{j-1}} := \bigcup_{\{ E_{j} \in \mathcal{T}_{j} : \  E_{j} \subset E_{j-1}\}} E_{j}.$ 
We introduce the  virtual element space $V_{j}^{E_{j-1}}$ given by a patch of local virtual element spaces ${V}^{E_{j}}$ where $E_{j}\in \mathcal{T}_{E_{j-1}}$, i.e.,
\begin{equation*}
V_{j}^{E_{j-1}} := \{ u \in H^{1}(E_{j-1}) \cap C^{0}(E_{j-1}): u_{{|_{E_{j}}}} \in {V}^{E_{j}}, \ E_{j} \in \mathcal{T}_{E_{j-1}} \}.
\end{equation*}

We denote by $\mathcal{N}(\mathcal{T}_{E_{j-1}}) := \cup_{E_{j} \in \mathcal{T}_{E_{j-1}}} \mathcal{N}(E_{j})$ the set of the indices of the nodes of all the elements $E_{j} \in \mathcal{T}_{E_{j-1}}$ and, finally, by $\mathcal{N}(\mathcal{T}_{E_{j-1}} \backslash E_{j-1}) := \mathcal{N}(\mathcal{T}_{E_{j-1}}) \backslash \mathcal{N}(E_{j-1})$ the set of the indices of the  nodes that belong to the elements $E_{j} \in \mathcal{T}_{E_{j}}$, but not to the element $E_{j-1}$. In Figure \ref{nodes_representation}, we provide a graphic example of the  different sets of nodes.

We choose $I^{j}_{j-1}$ as the operator locally defined as
\begin{equation}\label{prolongation_operator}
\begin{aligned}
& I^{j}_{j-1} u_{{j-1}_{|_{E_{j-1}}}} :=  \!\!
\sum_{i \in \mathcal{N}(E_{j-1})} \!\! \text{dof}_{i} (u_{j-1}) \varphi^{E_{j}}_{i}   + \!\! \sum_{i \in \mathcal{N}(\mathcal{T}_{E_{j-1}} \backslash E_{j-1})} \!\! \text{dof}_{i} (\Pi^{0}_{1,E_{j-1}} u_{j-1}) \varphi^{E_j}_{i},
\end{aligned}
\end{equation}
with $I^{j}_{j-1} u_{{j-1}_{|_{E_{j-1}}}} \in V_{j}^{E_{j-1}}$.

To better clarify the local construction of the prolongation operator, let us consider the example shown in Figure \ref{nodes_representation}. In this picture the coarse element  $E_{j-1}$ consists of six elements $E_{j}$.  Given the VEM function $u_{j-1} \in V_{j-1}$ restricted to the element $E_{j-1}$, i.e., $u_{{j-1}_{|_{E_{j-1}}}} \in V^{E_{j-1}}$, the prolongation operator gives the VEM function  $I^{j}_{j-1} u_{{j-1}_{|_{E_{j-1}}}} \in V_{j}^{E_{j-1}}$. As $I^{j}_{j-1} u_{{j-1}_{|_{E_{j-1}}}}$ is a VEM function on $ V_{j}^{E_{j-1}}$, then it is also a VEM function of the local virtual element space $V^{E_{j}}$ defined on each of the six elements $E_{j}$. Therefore, it is locally defined as the linear combination of the local VEM basis functions $\varphi^{E_j}_{i}, \ i \in \mathcal{N}(\mathcal{T}_{E_{j-1}})$. As coefficients of the linear combination we select the values assumed by $u_{j-1}$ in the nodes $x_{i}, i \in \mathcal{N}(E_{j-1})$ (squared nodes) and the values assumed by its local polynomial projection $\Pi^{0}_{1,E_{j-1}} u_{j-1}$ in the nodes $x_{i}, i \in \mathcal{N}(\mathcal{T}_{E_{j-1}} \backslash E_{j-1})$ (circular nodes).

\begin{figure}
\centering
\includegraphics[width=5cm]{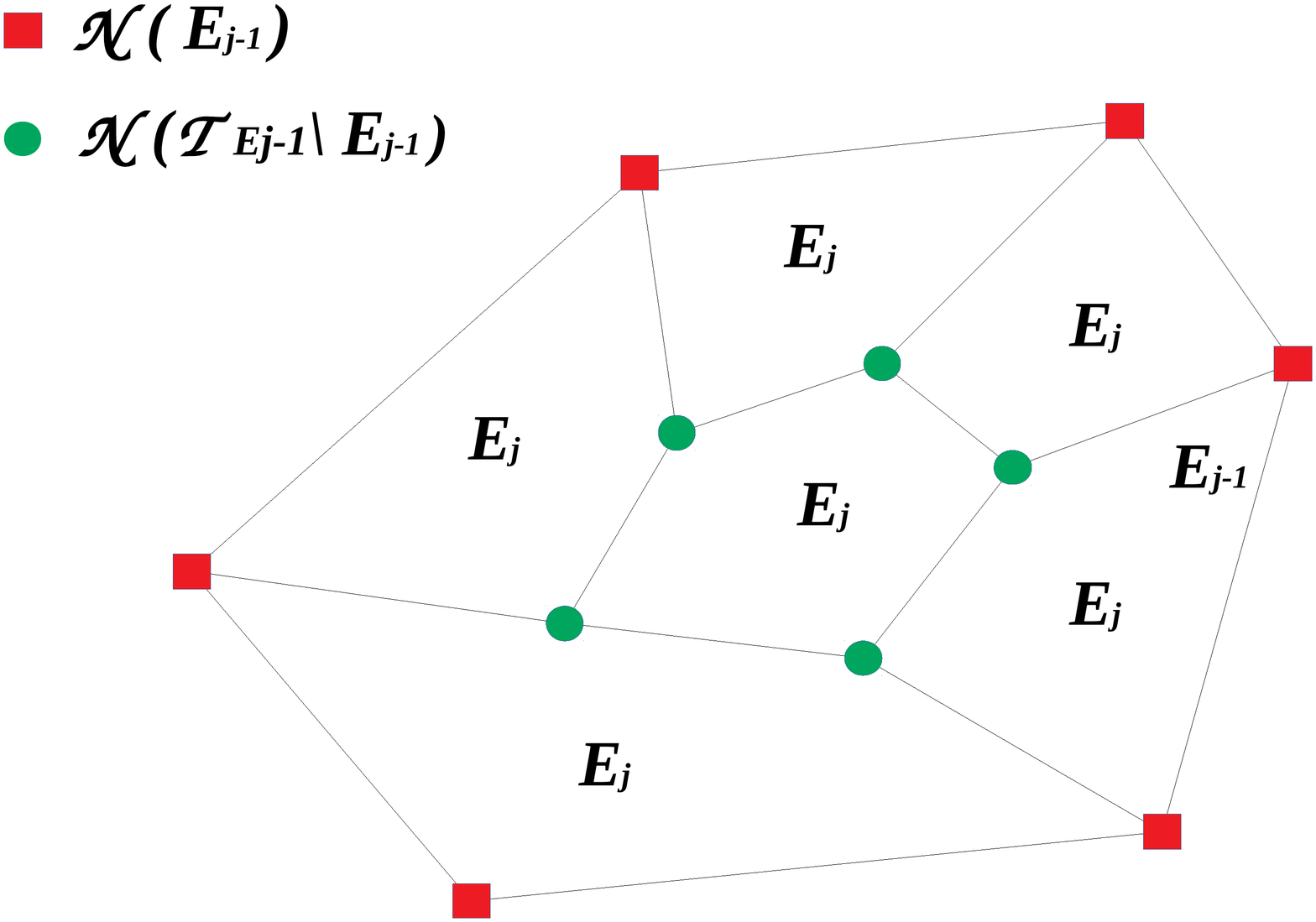}
\caption{Example of nodes related to the set of indices $\mathcal{N}(\mathcal{T}_{E_{j-1}} \backslash E_{j-1})$ (red squares) and to the set of indices $\mathcal{N}(E_{j-1})$ (green circles).}
\label{nodes_representation}
\end{figure}

\section{The BPX framework}\label{BPX_framework}
In the following section, we apply the BPX multigrid framework to the theoretical convergence analysis of our multigrid virtual element method. The BPX multigrid theory was firstly developed by Bramble, Pasciak and Xu in \cite{bramble1991analysis} for the analysis of multigrid methods with non-nested and non-inherited quadratic forms. Then, it was later extended in \cite{duan2007generalized}.

First, we introduce the assumptions that stands at the basis of the BPX theory and then we recall the theorems that guarantee the convergence of the method under these assumptions.

The BPX multigrid theory is based on the following assumptions.
\begin{assumption}\label{bilinearform_assumption}
$\exists C_{A3} > 0$ such that for any $j = 2, \dots, J$
\begin{equation*}
\mathcal{A}_{j}(I_{j-1}^{j} u, I_{j-1}^{j} u) \leq C_{A3} \ \mathcal{A}_{j-1}(u,u) \ \ \ \forall u \in V_{j-1},
\end{equation*}
where $C_{A3}$ is independent of $j$.
\end{assumption}

\begin{assumption}\label{regularity_approximation_assumption} 
Approximation property: $\exists C_{A4} > 0$ such that
\begin{equation}
| \mathcal{A}_{j}((I - I_{j-1}^{j} P_{j}^{j-1}) u,u )| \leq C_{A4} \frac{\|A_{j} u\|^{2}_{j}}{\lambda_j} \ \ \ \forall u \in V_{j}, \ j = 2 \dots, J,
\end{equation}
where $\lambda_j$ is the largest eigenvalue of $A_{j}$, $C_{A4}$ is independent of $j$, and $\| \cdot \|_j$ is the norm induced by $(\cdot,\cdot)_j$.
\end{assumption}

\begin{assumption}\label{smoothing_hypothesis}
Smoothing property: $ \exists C_{A5} > 0 $ such that
\begin{equation}\label{equation_smoothing_hypothesis}
\frac{\|u\|^{2}_{j}}{\lambda_j} \leq C_{A5} (\tilde{R}_{j}u,u)_{j} \ \ \ \forall u \in V_{j}, \ j = 1 \dots, J,
\end{equation}
where $\tilde{R}_{j} = (I - K_{j}^{*}K_{j})A_{j}^{-1}$ and $C_{A5}$ is independent of $j$.
\end{assumption}

The validity of Assumptions A\ref{bilinearform_assumption} and 
A\ref{smoothing_hypothesis} is proved in  \Cref{Convergence_analysis}.  Concerning Assumption A\ref{regularity_approximation_assumption}, in \cite{duan2007generalized} it has been proved that the following hypotheses are sufficient for the validity of Assumption A\ref{regularity_approximation_assumption}. In \Cref{Convergence_analysis}, we prove that Hypotheses H\ref{hypothesis2H}-H\ref{hypothesis8H} are satisfied in our framework involving the elliptic problem \eqref{elliptic_problem} satisfying the elliptic regularity assumption \eqref{elliptic_regularity}.

\begin{hypothesis} \label{hypothesis2H}
$\mathcal{A}_{j}(\cdot, \cdot): V_{j} \times V_{j} \to \mathbb{R}$ is a symmetric, positive definite and bounded bilinear form and we define
\begin{equation*}
\vertiii{u}_{1,j} := \sqrt{\mathcal{A}_j(u,u)} \ \ \ \forall u \in V_{j}, \ \forall j.
\end{equation*}
\end{hypothesis}

\begin{hypothesis} \label{hypothesis3H}
There exists an interpolation operator $\mathcal{I}^{i}: H^{2}(\Omega) \to V_{i}$ such that for all $j = 2, \dots, J$,

\begin{equation}\label{hypothesis3}
\| u - \mathcal{I}^{i} u \|_{L^{2}(\Omega)} + \delta_{i} \vertiii{u - \mathcal{I}^{i} u }_{1,i} \leq C \delta^{2}_{i} \| u \|_{H^{2}(\Omega)} \ \ \ i = j-1,j. 
\end{equation}
\end{hypothesis}

\begin{hypothesis}\label{hypothesis4H}
For all $j = 2, \dots, J$, it holds
\begin{equation}
\| I_{j-1}^{j} v \|_{L^{2}(\Omega)} \leq C_{H4} \  \| v \|_{L^{2}(\Omega)} \ \ \ \forall v \in V_{j-1}.
\end{equation}
\end{hypothesis}

\begin{hypothesis} \label{hypothesis5H}
For all $j = 1, \dots, J$, it holds 
\begin{equation} \label{equivalence_scalar_product}
C^{-1} \| v \|_{L^{2}(\Omega)} \leq \|v\|_{j} \leq C \|v\|_{L^{2}(\Omega)} \ \ \ \forall v \in V_{j}.
\end{equation}
\end{hypothesis}

\begin{hypothesis} \label{hypothesis6H}
For all $j = 1, \dots, J$, the following inverse inequality holds
\begin{equation}\label{hypothesis4}
\vertiii{v}_{1,j} \leq C \delta_{j}^{-1} \| v \|_{L^{2}(\Omega)} \ \ \ \forall v \in V_{j}.
\end{equation}
\end{hypothesis}

\begin{hypothesis} \label{hypothesis7H}
Let $f \in L^{2}(\Omega)$. Let $u \in V$ and $u_{i} \in V_{i}$ be respectively the solution of 
\begin{equation}\label{hypothesis7}
\mathcal{A}(u,u) = (f,v)_{L^{2}(\Omega)} \ \ \ \forall v \in V, \ \ \ \mathcal{A}_{i}(u_i,v) = (f,v)_{L^{2}(\Omega)} \ \ \ \forall v \in V_{i}.
\end{equation}
For all $j = 2, \dots, J$, we require that
\begin{equation*}
\| u - u_{i} \|_{L^{2}(\Omega)} + \delta_{i} \vertiii{ u - u_{i} }_{1,i} \leq C \delta_{i}^{2} \| f \|_{L^{2}(\Omega)} \ \ \ i = j-1, j.
\end{equation*}
\end{hypothesis}

\begin{hypothesis}\label{hypothesis8H}
The following estimate holds true
\begin{equation*}
\| \mathcal{I}^{j} w - I_{j-1}^{j} \mathcal{I}^{j-1} w \|_{L^2(\Omega)} \leq C_{H7} \  \delta_{j}^{2} \| w \|_{H^{2}(\Omega)} \ \ \ \forall w \in H^2(\Omega).
\end{equation*}
\end{hypothesis}

The convergence analysis of the multigrid method is stated in the following 
two theorems \cite{bramble1991analysis,duan2007generalized} that prove that under Assumptions A\ref{bilinearform_assumption}, A\ref{regularity_approximation_assumption} and  A\ref{smoothing_hypothesis}, the error propagation operator $I - B_j A_j$ defined in \eqref{error_propagator_operator} satisfies
\begin{equation}\label{convergence}
|\mathcal{A}_{j}((I -B_j A_j) u, u)| \leq \sigma \mathcal{A}_{j}(u,u) \ \ \ \forall u \in V_{j}, \ \ \ \forall j \geq 1,
\end{equation}
with constant $\sigma < 1$. In particular, \cref{V_cycle_convergence} and \cref{W_cycle_convergence} state the convergence of the symmetric V-cycle method and W-cycle methods, respectively. They include both the case in which the bilinear form $\mathcal{A}_{j-1}(\cdot, \cdot)$ is  inherited and non-inherited as in \cref{inherited_non_inherited_bilinear_form}.
\begin{theorem}\citetext{\citealp[Theorem 2]{bramble1991analysis}; \citealp[Theorem 3.1]{duan2007generalized}}
\label{V_cycle_convergence}
If Assumption A\ref{bilinearform_assumption} holds with $C_{A3} = 1$ and if Assumptions A\ref{regularity_approximation_assumption} and A\ref{smoothing_hypothesis} hold, then for the V-cycle multigrid ($p=1$) inequality \eqref{convergence} holds true with
\begin{equation}\label{delta}
\sigma = \frac{M}{M + \nu},
\end{equation}
where $M$ depends on $C_{A4}$ and $C_{A5}$, and $\nu$ is the number of smoothing steps.
Moreover, if Assumptions A\ref{bilinearform_assumption}, A\ref{regularity_approximation_assumption} and A\ref{smoothing_hypothesis} hold, then for the V-cycle multigrid ($p=1$) inequality \eqref{convergence} holds true with
\begin{equation*}
\sigma = \frac{C_{A4} C_{A5}}{\nu - C_{A4} C_{A5}},
\end{equation*}
provided that $\nu > 2 C_{A4} C_{A5}$.
\end{theorem}

\begin{theorem}{\cite[Theorems 3 and 7]{bramble1991analysis}}\label{W_cycle_convergence}
If Assumption A\ref{bilinearform_assumption} holds with $C_{A3} = 1$ and if Assumptions A\ref{regularity_approximation_assumption} and A\ref{smoothing_hypothesis} hold, then for the W-cycle multigrid ($p=2$) \eqref{convergence} holds true with $\sigma$ defined as in \eqref{delta}.
Moreover, if Assumptions A\ref{bilinearform_assumption}, A\ref{regularity_approximation_assumption} and A\ref{smoothing_hypothesis} hold, then for the W-cycle multigrid ($p=2$) \eqref{convergence} holds true with $\sigma$ defined as in \eqref{delta} provided that $\nu$ is chosen sufficiently large.
\end{theorem}

In the rest of the paper, we prove the validity of Hypotheses H\ref{hypothesis2H}-H\ref{hypothesis8H} and of Assumptions A\ref{bilinearform_assumption} and A\ref{smoothing_hypothesis} for the two-level method. Therefore, we set $J=2$ and we consider the two non-nested spaces $V_{J-1}$ and $V_{J}$. Next, we generalize the analysis to the V-cycle and the W-cycle.

\section{Convergence analysis}\label{Convergence_analysis}

In this section we prove the validity of all Hypotheses H\ref{hypothesis2H}-H\ref{hypothesis8H} and of Assumptions A\ref{bilinearform_assumption} and A\ref{smoothing_hypothesis}. In  \Cref{Convergence_analysis_two_level_method} we focus on the convergence of the two-level method and then, in  \Cref{uniformly_bounded_constant} we extend the results to the analysis of the convergence of the V-cycle and the W-cycle multigrid schemes.

\subsection{Convergence analysis of the two-level method}\label{Convergence_analysis_two_level_method}
Hypothesis H\ref{hypothesis2H} is satisfied since by construction the forms $\mathcal{A}_{j}(\cdot, \cdot)$ are symmetric, positive definite and bounded bilinear forms for all $j$.

We set $\vertiii{u}_{1,E_j} : = \sqrt{\mathcal{A}_{j}^{E_{j}}(u,u)}$. Therefore, 
\begin{equation*}
\vertiii{u}_{1,j}^{2} = \sum_{E_{j} \in \mathcal{T}_{j}} \vertiii{u}_{1, E_j}^{2} = \sum_{E_{j} \in \mathcal{T}_{j}} \mathcal{A}_{j}^{E_{j}}(u,u) = \mathcal{A}_{j}(u,u).
\end{equation*}

In particular, proceeding as in \cite{chen2018some},  it can be proved that for any $u \in {V}^{E_{j}}$, the following norm equivalence holds
\begin{equation*}
|  u |_{H^{1}(E_{j})}^2  \eqsim \mathcal{A}^{E_{j}}_{j}(u,u).
\end{equation*}
As a consequence,  we conclude that
\begin{equation}\label{equivalenceH1}
\vertiii{u}_{1,j} \eqsim |u|_{H^{1}(\Omega)},
\end{equation}
 and we will use this equivalence in the following proofs.
 
As interpolation operator, we consider the 
operator $\mathcal{I}^{j}: H^{2}(\Omega) \to V_{j}$ defined as 
\begin{equation}\label{interpolation_operator}
\text{dof}_{i}(u - \mathcal{I}^{j}u) = 0 \ \ \ \forall u \in V_{j}, \  i \in \mathcal{N}(\mathcal{T}_{j}).
\end{equation}
For the enhanced virtual element framework, Hypothesis H\ref{hypothesis3H} follows from the following proposition given in \cite{ahmad2013equivalent}.

\begin{proposition}\label{interpolant_estimates}
Assume that Assumption A\ref{mesh_assumption} is satisfied. Then, for  $E_{i} \in \mathcal{T}_{i}, \ i = j-1, j$ and for every $u \in H^2(E_{i})$, the interpolant $\mathcal{I}^{i} u \in V_{i}$ defined in \eqref{interpolation_operator} satisfies
\begin{equation}\label{interpolant_estimates_local}
\| u - \mathcal{I}^{i} u \|_{L^{2}(E_i)} + \delta_{i} \\|u - \mathcal{I}^{i} u \\|_{H^1(E_i)} \lesssim \delta^{2}_{i} | u |_{H^{2}(E_i)},
\end{equation}
the hidden constant depends on $\rho_{0}$ defined in Assumption A\ref{mesh_assumption}.
\end{proposition}

If we choose  $(\cdot,\cdot)_j$ as the $L^{2}$-inner product, then Hypothesis H\ref{hypothesis5H} is satisfied.  
In the following, we denote by $\mathcal{C}_{E_{j-1}}^{x_{i}}$ the set of elements $E_{j} \in \mathcal{T}_{E_{j-1}}$ having the node $x_i$ as vertex and by $\#\mathcal{C}_{E_{j-1}}^{x_{i}}$ its cardinality. 
\begin{proposition}\cite[Corollary 4.6]{chen2018some} \label{norm_equivalence}
For any $u \in {V}^{E_j}, \  E_j \in \mathcal{T}_{j}$, the following norm equivalence holds true
\begin{equation}\label{norms}
\delta_{j} \sqrt{ \sum_{i \in \mathcal{N}(E_{j})} \big (\text{dof}_{i}(u) \big)^2}
 \lesssim \| u \|_{L^{2}(E_j)} \lesssim \delta_{j} \sqrt{ \sum_{i \in \mathcal{N}(E_{j})} \big (\text{dof}_{i}(u) \big)^2}.
\end{equation}
Moreover, for any $u \in V_{j}^{E_{j-1}}$, $E_{j-1} \in \mathcal{T}_{j-1}$, the following norm equivalence holds
\begin{equation*}
\delta_{j}^2 \sum_{i \in \mathcal{N}(\mathcal{T}_{E_{j-1}})} \# \mathcal{C}_{E_{j-1}}^{x_i} |\text{dof}_{i}(u)|^{2} \lesssim \| u \|^{2}_{L^{2}(E_{j-1})} \lesssim \delta_{j}^2 \sum_{i \in \mathcal{N}( \mathcal{T}_{E_{j-1}})} \# \mathcal{C}_{E_{j-1}}^{x_i} |\text{dof}_{i}(u)|^{2}.
\end{equation*}
\end{proposition} 

Hypothesis H\ref{hypothesis6H} can be proved from the inverse inequality of a VEM function reported in the following theorem \cite{chen2018some}.
\begin{theorem}\cite[Theorem 3.6]{chen2018some}\label{theorem_inverse_inequality}
The following inverse inequality holds true
\begin{equation}\label{inverse_inequality_equation}
\| \nabla u \|_{L^{2}(E_j)} \lesssim \delta_{j}^{-1} \| u \|_{L^{2}(E_j)} \ \ \ \forall v \in {V}^{E_{j}}.
\end{equation}
\end{theorem}

Assuming $f \in H^{1}(\Omega)$, Hypothesis H\ref{hypothesis7H} results from   \eqref{equivalenceH1} and the following theorem reported in \cite{ahmad2013equivalent}.
\begin{theorem}\cite[Theorem 3]{ahmad2013equivalent}
Let $u$ be the solution to the problem  $\mathcal{A}(u,u) = (f,v)_{L^{2}(\Omega)}, \ \forall v \in V$ and let $u_{i} \in V_{i}, \ i = j-1, j$ be the solution to the discrete problem $\mathcal{A}_{i}(u_i,v) = (f,v)_{L^{2}(\Omega)}, \ \forall v \in V_{i}$. Assume further that $\Omega$ is convex, that the right-hand side $f$ belongs to $H^{1}(\Omega)$, and that the exact solution $u$ belongs to $H^{2}(\Omega)$. Then the following estimate holds true
\begin{equation*}
\|u - u_{i} \|_{L^{2}(\Omega)} + \delta_i \|u - u_{i} \|_{H^{1}(\Omega)} \lesssim \delta_{i}^{2} |u|_{H^{2}(\Omega)},
\end{equation*}
where the hidden constant is independent of $\delta_i$. 
\end{theorem}
 
\begin{remark}
We underline that with respect to the standard BPX theory, we need to require $f \in H^{1}(\Omega)$ in order to have Hypothesis H\ref{hypothesis7H} satisfied.
\end{remark}

Now, we show that the stability result of the prolongation operator $I_{j-1}^{j}$ in Hypothesis H\ref{hypothesis4H} holds true.
\begin{proposition}\label{stability_prolongation_operator}
Let $I^{j}_{j-1}$ be the prolongation operator defined as in \eqref{prolongation_operator}. The stability estimate Hypothesis H\ref{hypothesis4H} holds true.
\end{proposition}
\begin{proof} To begin with, let us focus on the element $E_{j-1} \in \mathcal{T}_{j-1}$. We show that 
\begin{equation*}
\| I_{j-1}^{j} u_{j-1} \|_{L^{2}(E_{j-1})} \lesssim \| u_{j-1} \|_{L^{2}(E_{j-1})} \ \ \ \forall u_{j-1} \in V_{j-1}.
\end{equation*}
We apply \cref{norm_equivalence}  to $\| I_{j-1}^{j} u_{j-1} \|_{L^{2}(E_{j-1})}$ and we use the definition of the prolongation operator $I_{j-1}^{j}$. Moreover, we define $\#\mathcal{C}_{E_{j-1}} := \max_{i \in \mathcal{N}(\mathcal{T}_{E_{j-1}})} \# \mathcal{C}^{x_i}_{E_{j-1}}$
\begin{equation*}
\begin{aligned}
&\| I_{j-1}^{j} u_{j-1} \|_{L^{2}(E_{j-1})}^2 \lesssim \delta_{j}^{2} \sum_{i \in \mathcal{N}(\mathcal{T}_{E_{j-1}})} \# \mathcal{C}^{x_i}_{E_{j-1}} | \text{dof}_{i} (I_{j-1}^{j} u_{j-1})|^{2} \\
& = \delta_{j}^{2} \sum_{i \in \mathcal{N}(\mathcal{T}_{E_{j-1}})} \# \mathcal{C}^{x_i}_{E_{j-1}} | I_{j-1}^{j} u_{j-1}(x_i)|^{2}  \lesssim \delta_{j}^{2}  \# \mathcal{C}_{E_{j-1}} \sum_{i \in \mathcal{N}(\mathcal{T}_{E_{j-1}})}  | I_{j-1}^{j} u_{j-1}(x_i)|^{2}  \\
&= \delta_{j}^{2}  \# \mathcal{C}_{E_{j-1}} \Big (\sum_{i \in \mathcal{N}(E_{j-1})} |u_{j-1}(x_i)|^2  + \sum_{i \in \mathcal{N}(\mathcal{T}_{E_{j-1}} \backslash E_{j-1})} |\Pi^{0}_{1, E_{j-1}} u_{j-1}(x_i)|^{2} \Big ).
\end{aligned}
\end{equation*}

Next, we bound each of the two terms on the right-hand side separately. For the first one, we apply \cref{norm_equivalence} for $u_{j-1} \in {V}^{E_{j-1}}, E_{j-1} \in \mathcal{T}_{j-1}$ to obtain
\begin{equation} \label{estimate1}
\begin{aligned}
& \sum_{i \in \mathcal{N}(E_{j-1})} |u_{j-1}(x_i)|^2  \lesssim  \frac{1}{\delta_{j-1}^2}  \| u_{j-1}\|_{L^{2}(E_{j-1})}^2.
\end{aligned}
\end{equation}
For the second one, firstly, we add positive quantities and then we make use of \cref{norm_equivalence} and of the $L^{2}(E_{j-1})$-stability of the projection operator $\Pi^{0}_{1, E_{j-1}}$
\begin{equation}\label{estimate2}
\begin{aligned}
&  \sum_{i \in \mathcal{N}(\mathcal{T}_{E_{j-1}} \backslash E_{j-1})} |\Pi^{0}_{1, E_{j-1}} u_{j-1}(x_i)|^{2}   \lesssim \sum_{i \in \mathcal{N}(\mathcal{T}_{E_{j-1}})} |\Pi^{0}_{1, E_{j-1}} u_{j-1}(x_i)|^{2} \\
&  \lesssim  \sum_{i \in \mathcal{N}(\mathcal{T}_{E_{j-1}})} \# \mathcal{C}^{x_i}_{E_{j-1}} |\Pi^{0}_{1, E_{j-1}} u_{j-1}(x_i)|^{2}   \lesssim \frac{1}{\delta_j^2} \|\Pi^{0}_{1, E_{j-1}} u_{j-1} \|^2_{L^2(E_{j-1})}  \lesssim  \frac{1}{\delta_j^2} \| u_{j-1}\|_{L^{2}(E_{j-1})}^2.
\end{aligned}
\end{equation}
Estimate \eqref{estimate1} together with estimate \eqref{estimate2} leads to
\begin{equation*}
\| I_{j-1}^{j} u_{j-1} \|_{L^{2}(E_{j-1})}^2 \lesssim \Big [ \Big ( \frac{\delta_{j}}{\delta_{j-1}} \Big)^2 + 1 \Big ] \# \mathcal{C}_{E_{j-1}}  \| u_{j-1}\|_{L^{2}(E_{j-1})}^2.
\end{equation*}
Finally, summing on all $E_{j-1} \in \mathcal{T}_{j-1}$, we obtain
\begin{equation*}
\| I_{j-1}^{j} u_{j-1} \|_{L^{2}(\Omega)} \leq C_{H4}  \| u_{j-1}\|_{L^{2}(\Omega)},
\end{equation*}
where $C_{H4} = C_{H4} \Big (\frac{\delta_{j}}{\delta_{j-1}}, \# \mathcal{C}_{E_{j-1}} \Big )$ and the proof is complete.
\end{proof}

In order to verify Hypothesis H\ref{hypothesis8H} we first prove the following preliminary results. 

\begin{lemma}\label{lemma_for_hypothesis_8}
For any $u_{j-1} \in V_{j}^{E_{j-1}}, \ E_{j-1} \in \mathcal{T}_{j-1}$, the following estimate holds
\begin{equation}\label{lemma1}
\| \Pi^{0}_{1, E_{j-1}} u_{j-1} - I_{j-1}^{j} u_{j-1} \|_{L^{2}(E_{j-1})} \lesssim \| \Pi^{0}_{1, E_{j-1}} u_{j-1} - u_{j-1} \|_{L^{2}(E_{j-1})}, 
\end{equation}
where the hidden constant depends on $\frac{\delta_{j}}{\delta_{j-1}}$ and $\# \mathcal{C}_{E_{j-1}}$.
\end{lemma}
\begin{proof}
To begin with, we apply \cref{norm_equivalence} to the left hand side term of \eqref{lemma1} and then we use the definition of the prolongation operator $I_{j-1}^{j}$ 
\begin{equation}\label{lemma1_1}
\begin{aligned}
& \| \Pi^{0}_{1, E_{j-1}} u_{j-1} - I_{j-1}^{j} u_{j-1}\|^2_{L^{2}(E_{j-1})}  \lesssim  \delta_{j}^{2} \sum_{i \in\mathcal{N}(\mathcal{T}_{E_{j-1}})} \# \mathcal{C}^{x_i}_{E_{j-1}} | \text{dof}_{i}( \Pi^{0}_{1, E_{j-1}} u_{j-1} - I_{j-1}^{j} u_{j-1})|^2 \\
& = \delta_{j}^{2} \sum_{i \in\mathcal{N}(\mathcal{T}_{E_{j-1}})} \# \mathcal{C}^{x_i}_{E_{j-1}} | \Pi^{0}_{1, E_{j-1}} u_{j-1}(x_i) - I_{j-1}^{j} u_{j-1}(x_i) |^2 \\
& \lesssim \delta_{j}^{2} \max_{i \in \mathcal{N}(\mathcal{T}_{E_{j-1}})} \# \mathcal{C}^{x_i}_{E_{j-1}} \sum_{i \in\mathcal{N}(\mathcal{T}_{E_{j-1}})}  | \Pi^{0}_{1, E_{j-1}} u_{j-1}(x_i) - I_{j-1}^{j} u_{j-1}(x_i) |^2 \\
& \lesssim \delta_{j}^{2} \# \mathcal{C}_{E_{j-1}} \Big ( \sum_{i \in \mathcal{N}(E_{j-1})}  | \Pi^{0}_{1, E_{j-1}} u_{j-1}(x_i) - u_{j-1}(x_i)|^2 \\
& +  \sum_{i \in\mathcal{N}(\mathcal{T}_{E_{j-1}} \backslash E_{j-1})}  | \Pi^{0}_{1, E_{j-1}} u_{j-1}(x_i) - \Pi^{0}_{1, E_{j-1}} u_{j-1}(x_i)|^2 \Big ).
\end{aligned}
\end{equation}
The second term of the last inequality of \eqref{lemma1_1} is zero. Therefore, we only need to estimate the first term. Using \cref{norm_equivalence}, we obtain
\begin{equation*}
\begin{aligned}
& \delta_{j}^{2} \# \mathcal{C}_{E_{j-1}}  \sum_{i \in \mathcal{N}(E_{j-1})} | \Pi^{0}_{1, E_{j-1}} u_{j-1}(x_i) - u_{j-1}(x_i)|^2 \\
& \lesssim  \# \mathcal{C}_{E_{j-1}} \Big (\frac{\delta_{j}}{\delta_{j-1}} \Big )^2 \| u_{j-1} - \Pi^{0}_{1, E_{j-1}} u_{j-1} \|_{L^{2}(E_{j-1})}^2  \lesssim \| u_{j-1} - \Pi^{0}_{1, E_{j-1}} u_{j-1} \|_{L^{2}(E_{j-1})}^2,
\end{aligned}
\end{equation*}
where the hidden constant depends on $\frac{\delta_{j}}{\delta_{j-1}}$ and  $\# \mathcal{C}_{E_{j-1}}$.
\end{proof}

\begin{lemma}\label{lemma_for_hypothesis_8_2}
For any $w \in H^{2}(E_{j-1}), \ E_{j-1} \in \mathcal{T}_{j-1}$, the following estimate holds
\begin{equation}\label{proposition2_1_3_bis}
\| \Pi^{0}_{1, E_{j-1}} w - \Pi^{0}_{1, E_{j-1}} \mathcal{I}^{j-1} w \|_{L^{2}(E_{j-1})}^2  \lesssim  \delta_{j}^{4} \| w \|_{H^{2}(E_{j-1})}^2,
\end{equation}
where the hidden constant depends on  $\frac{\delta_{j-1}}{\delta_{j}}$.
\end{lemma}
\begin{proof}
First, adding and subtracting $w - \mathcal{I}^{j-1}w$ and applying the triangle inequality yields
\begin{equation}\label{proposition2_1_3}
\begin{aligned}
& \| \Pi^{0}_{1, E_{j-1}} w - \Pi^{0}_{1, E_{j-1}} \mathcal{I}^{j-1} w \|_{L^{2}(E_{j-1})} \leq \| w - \mathcal{I}^{j-1}w \|_{L^{2}(E_{j-1})} \\
&  + \| (I - \Pi^{0}_{1, E_{j-1}}) ( w - \mathcal{I}^{j-1}w) \|_{L^{2}(E_{j-1})}.
\end{aligned}
\end{equation}
Next, we bound each of the two terms on the right-hand side of \eqref{proposition2_1_3} separately. For the first term, since $w \in H^{2}(E_{j-1})$, we can apply \cref{interpolant_estimates}. Then
\begin{equation} \label{proposition2_2_1}
\begin{aligned}
& \| \mathcal{I}^{j-1} w - w \|_{L^{2} (E_{j-1})}^2  \lesssim  \delta_{j-1}^{4} \|w \|_{H^{2}(E_{j-1})}^{2}  \lesssim \Big (\frac{\delta_{j-1}}{\delta_{j}} \Big)^4 \delta_{j}^{4} \| w \|_{H^{2}(E_{j-1})}^2 \lesssim  \delta_{j}^{4} \| w \|_{H^{2}(E_{j-1})}^2.
\end{aligned}
\end{equation}
For the second term, we notice that $\mathcal{I}^{j-1} w \in V_{j-1} \subset H^{1}(\Omega)$ and $w \in H^{2}(\Omega)$. Therefore, $w - \mathcal{I}^{j-1} w\in H^{1}(E_{j-1})$. Consequently, we can apply \cref{lemma_projector}
\begin{equation}\label{proposition2_2_2}
\| (I - \Pi^{0}_{1, E_{j-1}}) ( w - \mathcal{I}^{j-1}w) \|_{L^{2}(E_{j-1})}^2 \lesssim \delta_{j-1}^2 | w - \mathcal{I}^{j-1} w |_{H^{1}(E_{j-1})}^2.
\end{equation}
Since $w \in H^{2}(E_{j-1})$, we can apply \cref{interpolant_estimates} and we obtain
\begin{equation}\label{proposition2_2_2_bis}
\begin{aligned}
& \delta_{j-1}^2 | w - \mathcal{I}^{j-1} w |_{H^{1}(E_{j-1})^2}^2 \lesssim \delta_{j-1}^{4} \|w \|_{H^{2}(E_{j-1})}^2  \lesssim \Big (\frac{\delta_{j-1}}{\delta_{j}} \Big)^4 \delta_{j}^{4} \| w \|_{H^{2}(E_{j-1})}^2 \lesssim \delta_{j}^{4} \| w \|_{H^{2}(E_{j-1})}^2.
\end{aligned}
\end{equation}
Combining estimates \eqref{proposition2_2_1}, \eqref{proposition2_2_2} and \eqref{proposition2_2_2_bis} leads to \eqref{proposition2_1_3_bis}.
\end{proof}

Using the previous lemmata, we prove that the following holds true. 
\begin{proposition}
Let $\mathcal{I}^{j}$ be the interpolation operator defined in \eqref{interpolation_operator}. Then, Hypothesis H\ref{hypothesis8H} holds true.
\end{proposition}
\begin{proof} Let us focus on the element $E_{j-1} \in \mathcal{T}_{j-1}.$ We want to show that
\begin{equation}\label{proposition2}
\| \mathcal{I}^{j} w - I_{j-1}^{j} \mathcal{I}^{j-1} w \|_{L^{2}(E_{j-1})} \lesssim \delta_{j}^{2} \| w \|_{H^{2}(E_{j-1})} \ \ \ \forall w \in H^{2}(\Omega).
\end{equation}
By adding and subtracting  $w - \Pi^{0}_{1, E_{j-1}} w +  \Pi^{0}_{1, E_{j-1}} \mathcal{I}^{j-1} w $ and applying the triangle inequality, we obtain
\begin{equation}\label{proposition2_1}
\begin{aligned}
&\| \mathcal{I}^{j} w - I_{j-1}^{j} \mathcal{I}^{j-1} w \|_{L^{2}(E_{j-1})} \leq || \mathcal{I}^{j} w - w \|_{L^{2} (E_{j-1})} \\
&+ \| w - \Pi^{0}_{1, E_{j-1}} w \|_{L^{2}(E_{j-1})} + \| \Pi^{0}_{1, E_{j-1}} w - \Pi^{0}_{1, E_{j-1}} \mathcal{I}^{j-1} w \|_{L^{2}(E_{j-1})} \\
& + \| \Pi^{0}_{1, E_{j-1}} \mathcal{I}^{j-1} w - I_{j-1}^{j} \mathcal{I}^{j-1} w \|_{L^2(E_{j-1})}.
\end{aligned}
\end{equation}

In order to estimate the first term on the right-hand side of \eqref{proposition2_1}, we use \cref{interpolant_estimates} to obtain
\begin{equation}\label{proposition2_1_1}
\begin{aligned}
& \| \mathcal{I}^{j} w - w \|_{L^{2} (E_{j-1})}^2 = \sum_{E_{j} \in \mathcal{T}_{E_{j-1}}} \| \mathcal{I}^{j} w - w \|_{L^{2}(E_{j})}^{2} \lesssim \sum_{E_{j} \in \mathcal{T}_{E_{j-1}}} \delta_{j}^{4} |w|_{H^{2}(E_{j})}^2 \\
& \lesssim \sum_{E_{j} \in \mathcal{T}_{E_{j-1}}} \delta_{j}^{4} \|w \|_{H^{2}(E_{j})}^{2} \lesssim \delta_{j}^{4} \| w \|_{H^{2}(E_{j-1})}^2.
\end{aligned}
\end{equation}
To estimate the second term on the right-hand side of \eqref{proposition2_1}, we use \cref{lemma_projector}
\begin{equation}\label{proposition2_1_2}
\begin{aligned}
& \| w - \Pi^{0}_{1, E_{j-1}} w \|_{L^{2}(E_{j-1})}^2 = \sum_{E_{j} \in \mathcal{T}_{E_{j-1}}} \| w - \Pi^{0}_{1, E_{j-1}} w \|_{L^{2}(E_{j})}^2 \\
& \lesssim \sum_{E_{j} \in \mathcal{T}_{E_{j-1}}} \delta_{j}^{4} \|w \|_{H^{2}(E_{j})}^{2} \lesssim \delta_{j}^{4} \| w \|_{H^{2}(E_{j-1})}^2.
\end{aligned}
\end{equation}

To estimate the third term on the right-hand side of \eqref{proposition2_1}, we use  \cref{lemma_for_hypothesis_8_2}
\begin{equation}\label{proprosition2_1_3_new}
\| \Pi^{0}_{1, E_{j-1}} w - \Pi^{0}_{1, E_{j-1}} \mathcal{I}^{j-1} w \|_{L^{2}(E_{j-1})}^2  \lesssim \delta_{j}^{4} \| w \|_{H^{2}(E_{j-1})}^2.
\end{equation}

It remains to estimate the fourth term on the right-hand side of \eqref{proposition2_1}. Firstly, we apply \cref{lemma_for_hypothesis_8}, then we add and subtract the term $\Pi^{0}_{1, E_{j-1}} w - w$ and we apply the triangle inequality, to obtain
\begin{equation}\label{proposition2_1_4}
\begin{aligned}
& \| \Pi^{0}_{1, E_{j-1}} \mathcal{I}^{j-1} w - I_{j-1}^{j}  \mathcal{I}^{j-1} w \|_{L^{2}(E_{j-1})} \lesssim \| \Pi^{0}_{1, E_{j-1}}  \mathcal{I}^{j-1} w  -  \mathcal{I}^{j-1} w  \|_{L^{2}(E_{j-1})} \\
& \lesssim \| \Pi^{0}_{1, E_{j-1}} \mathcal{I}^{j-1} w - \Pi^{0}_{1, E_{j-1}} w \|_{L^{2}(E_{j-1})} + \| \Pi^{0}_{1, E_{j-1}} w - w \|_{L^{2}(E_{j-1})}  + \| \mathcal{I}^{j-1} w - w \|_{L^{2}(E_{j-1})}.
\end{aligned}
\end{equation}
An estimate for the first term on the right hand side of \eqref{proposition2_1_4} is provided in \cref{lemma_for_hypothesis_8_2}, whereas for the second term we use \cref{lemma_projector} as done in \eqref{proposition2_1_2} and for the third term we use \cref{interpolant_estimates}. Therefore, we obtain
\begin{equation}\label{proposition2_1_4_bis}
\begin{aligned}
& \| \Pi^{0}_{1, E_{j-1}} \mathcal{I}^{j-1} w - I_{j-1}^{j}  \mathcal{I}^{j-1} w \|_{L^{2}(E_{j-1})} \lesssim \delta_{j}^{2} \|w\|_{H^{2}(E_{j-1})}.
\end{aligned}
\end{equation}
Combining \eqref{proposition2_1_1}, \eqref{proposition2_1_2}, \eqref{proprosition2_1_3_new} and \eqref{proposition2_1_4_bis}, we obtain \eqref{proposition2}. Finally, summing on all $E_{j-1} \in \mathcal{T}_{j-1}$, we obtain the thesis with constant $C_{H7} = C_{H7}   \Big (\frac{\delta_{j}}{\delta_{j-1}}, \frac{\delta_{j-1}}{\delta_{j}}, \# \mathcal{C}_{E_{j-1}} \Big ).$
\end{proof}

If the bilinear form $\mathcal{A}_{j-1}(\cdot, \cdot)$ is inherited, cf. \cref{inherited_non_inherited_bilinear_form}, then Assumption A\ref{bilinearform_assumption} is trivially satisfied with $C_{A3} = 1$. Consequently, it remains to prove the following proposition. 
\begin{proposition} Let $\mathcal{A}_{j-1}(\cdot, \cdot)$ be the non-inherited bilinear form defined as in \cref{inherited_non_inherited_bilinear_form}. Then, Assumption A\ref{bilinearform_assumption} holds true.
\end{proposition}
\begin{proof}
Let $\bar{u}_{E_{j-1}}$ be the mean value of $u$ on $E_{j-1}$.  By the continuity of $\mathcal{A}_{j}(\cdot, \cdot)$ and noticing that $I^{j}_{{j-1}} \bar{u}_{E_{{j-1}_{|_{E_{j-1}}}}} = \bar{u}_{E_{j-1}}$, we obtain
\begin{equation*}
\begin{aligned}
& \mathcal{A}_{j}(I_{j-1}^{j}u, I_{j-1}^{j} u) = \sum_{E_{j} \in \mathcal{T}_{j}} \mathcal{A}_{j}^{E_{j}}(I_{j-1}^{j}u, I_{j-1}^{j} u) \lesssim \sum_{E_j \in \mathcal{T}_{j}} \|\nabla I_{j-1}^{j} u \|_{L^{2}(E_{j})}^2 \\
&  \lesssim \sum_{E_{j} \in \mathcal{T}_{j}} \| \nabla (I^{j}_{j-1} u - \bar{u}_{E_{j-1}}) \|^{2}_{L^{2}(E_j)}  \lesssim \sum_{E_{j} \in \mathcal{T}_{j}} \|\nabla I_{j-1}^{j}(u-\bar{u}_{E_{j-1}})\|^2_{L^{2}(E_j)}.
\end{aligned}
\end{equation*}
Next, we apply the inverse inequality \cref{inverse_inequality_equation} for VEM function on $V^{E_{j}}$ and the stability of $I^{j}_{j-1}$ in $L^{2}(E_{j-1})$ proved in \cref{stability_prolongation_operator}
\begin{equation*}
\begin{aligned}
&\sum_{E_{j} \in \mathcal{T}_{j}} \|\nabla I_{j-1}^{j}(u-\bar{u}_{E_{j-1}})\|^2_{L^{2}(E_j)} \lesssim \sum_{E_{j} \in \mathcal{T}_{j}} \frac{1}{\delta_{j}^2} \| I_{j-1}^{j} (u - \bar{u}_{E_{j-1}}) \|^{2}_{L^{2}(E_{j})} \\
& = \sum_{E_{j-1} \in \mathcal{T}_{j-1}} \sum_{E_{j} \in \mathcal{T}_{E_{j-1}}} \frac{1}{\delta_{j}^{2}} \|I_{j-1}^{j}(u-\bar{u}_{E_{j-1}})\|^{2}_{L^{2}(E_{j})} \\
& = \frac{1}{\delta_{j}^{2}} \sum_{E_{j-1} \in \mathcal{T}_{j-1}} \|I_{j-1}^{j}(u - \bar{u}_{E_{j-}})\|^{2}_{L^{2}(E_{j-1})} \lesssim \frac{1}{\delta_{j}^{2}} \sum_{E_{j-1} \in \mathcal{T}_{j-1}} \| u - \bar{u}_{E_{j-1}}\|^{2}_{L^{2}(E_{j-1})}.
\end{aligned}
\end{equation*}
Finally, we apply the Poincar\'e inequality to the virtual element function $u - \bar{u}_{E_{j-1}}$ that has zero mean on $E_{j-1}$ by definition of $\bar{u}_{E_{j-1}}$ and we conclude by the coercivity of $\mathcal{A}_{j-1}(\cdot,\cdot)$
\begin{equation*}
\begin{aligned}
& \frac{1}{\delta_{j}^{2}} \sum_{E_{j-1} \in \mathcal{T}_{j-1}} \| u - \bar{u}_{E_{j-1}}\|^{2}_{L^{2}(E_{j-1})} \lesssim \Big (\frac{\delta_{j-1}}{\delta_{j}} \Big)^2 \sum_{E_{j-1}\in \mathcal{T}_{j-1}} \| \nabla (u - c_{E_{j-1}}) \|^{2}_{L^{2}(E_{j-1})} \\
& = \Big( \frac{\delta_{j-1}}{\delta_{j}}\Big)^2 \sum_{E_{j-1}\in \mathcal{T}_{j-1}} \| \nabla u \|^{2}_{L^{2}(E_{j-1})} 
\lesssim \Big( \frac{\delta_{j-1}}{\delta_{j}}\Big)^2 \mathcal{A}_{j-1}(u,u).
\end{aligned}
\end{equation*}
Therefore, Assumption A\ref{bilinearform_assumption} holds true with constant $C_{A3} = C_{A3} \Big (\frac{\delta_{j-1}}{\delta_{j}} \Big )$.
\end{proof}

We prove Assumption A\ref{smoothing_hypothesis} relying on the abstract results reported in \cite{bramble1992analysis} for smoothing operators defined in terms of subspace decomposition such as Parallel Subspace Correction (PSC) and Successive Subspace Correction (SSC). Indeed, the Gauss-Seidel method can be interpreted as a SSC method. Let us consider the following decomposition of the global virtual element space $V_{j}$ defined in \cref{global_VEM}
\begin{equation}\label{V_decomposition}
V_{j} = \sum_{i = 1}^{\mathcal{N}^{j}_{dof}} V_{j}^{i},
\end{equation}
where $V_{j}^{i} := \text{span} \{ \varphi^{j}_{i} \}$. Moreover, let $A_{j,i} : V_{j}^{i} \to V_{j}^{i}$ be defined by $(A_{j,i} v, u)_{j} = (A_{j} v, u)_{j} \ \forall v \in V_{j}^{i}$,
and $Q_{j}^{i}: V_{j} \to V^{i}_{j}$ be the projection onto $V_{j}^{i}$ with respect to the inner product $(\cdot,\cdot)_{j}$. Let $w \in V_{j}$. Given the subspace decomposition \eqref{V_decomposition} of $V_{j}$, the SSC operator $R_{j}: V_{j} \to V_{j}$ is defined in \cref{SSC}. 

\begin{algorithm}
\footnotesize
\caption{Successive subspace correction method (SSC) \ \ \ $R_{j} w = \text{SSC}(j,w)$}\label{SSC}
\begin{enumerate}
\item Set $v_0 = 0$.
\item Define $v_{i}$ for $i = 1, \dots, \mathcal{N}^{j}_{dof}$ by
\begin{equation*}
v_{i} = v_{i-1} + A_{j,i}^{-1} Q_{j}^{i}(w - A_{j}v_{i-1}).
\end{equation*}
\item Set $R_{j} w = v_{l}$.
\end{enumerate}
\end{algorithm}

In \cite{bramble1992analysis}, it is shown that Assumption A\ref{smoothing_hypothesis} holds for $R_{j}$ defined as in \cref{SSC}.
\begin{theorem}\cite[Theorem 3.2]{bramble1992analysis}\label{Smoothing_theorem}
Let $R_{j}$ be defined as in \cref{SSC} and let the projection $P^{i}_{j}: V_{j} \to V_{j}^{i}$ be defined by
\begin{equation*}
(A_{j} P^{i}_{j} v, u)_{j} = (A_{j} v, u)_{j} \ \ \  \forall u \in V_{j}^{i}.
\end{equation*}
Moreover, define $\kappa_{im} = 0 $ if $ P_{j}^{i} P_{j}^{m} = 0$ and equal to $1$ otherwise, 
and set $n_0 = \text{max}_{i} \sum_{m=1}^{\mathcal{N}^{j}_{dof}} \kappa_{im}.$ Assume that the following two conditions hold:
\begin{enumerate}
\item The subspaces satisfy a limited interaction property, i.e., $n_0 \leq c_1,$ with $c_1$ independent of $j$.
\item There exists a positive constant $c_0$ not depending on $j$ such that for each $u \in V_{j}$ there is a decomposition $u = \sum_{i = 1}^{\mathcal{N}^{j}_{dof}} u_{i}$ with $u_{i} \in V_{j}^{i}$ satisfying
\begin{equation*}
\sum_{i=1}^{\mathcal{N}^{j}_{dof}} || u_{i}||_{j}^{2} \leq c_{0} ||u||_{j}^{2}.
\end{equation*}
\end{enumerate}
Then \eqref{equation_smoothing_hypothesis} holds with
\begin{equation}\label{constant_A4}
C_{A5} = 2 c_0 (1 + c_1^2).
\end{equation}
\end{theorem}
In our particular case, it turns out that $\kappa_{im}$ is different from zero only if $\Omega_{j}^{i} \cap \Omega_{j}^{m} \neq \emptyset$, where we denote by $\Omega_{j}^{i}$ the support of the Lagrangian basis function $\varphi^{j}_{i}, i = 1, \dots, \mathcal{N}^{j}_{dof}$. Consequently,  we can take $c_1$ as the maximum number of supports $\{\Omega_{j}^{m}\}$ of the basis functions $\{\varphi^{j}_{m}\}$ that intersect $\Omega_{j}^{i}$. Due to the mesh regularity requirements of Assumption A\ref{mesh_assumption}, $c_1$ is a bounded quantity. Moreover, we can set $c_0 = 1$. Therefore, the two conditions stated in \cref{Smoothing_theorem} are satisfied and we conclude that Assumption A\ref{smoothing_hypothesis} holds with $C_{A5}$ defined as in \eqref{constant_A4} in case we choose $R_{j}$ to be the linear smoothing operator induced by to the Gauss-Seidel smoother.

\subsection{Convergence analysis of the V-cycle and W-cycle methods} \label{uniformly_bounded_constant}
In this section we briefly deal with the convergence of V-cycle and W-cycle, i.e., when $J>2$, by generalizing the proof of the convergence of the two-level method. To this aim, let us first remark that a closer inspection to the proofs of 
Hypotheses H\ref{hypothesis3H}, H\ref{hypothesis5H}, H\ref{hypothesis6H} and H\ref{hypothesis7H} reveals that the constants appearing in \eqref{hypothesis3}, \eqref{equivalence_scalar_product}, \eqref{hypothesis4} and  \eqref{hypothesis7} depend on the considered level $j$. Moreover, the constants $C_{A3}$, $C_{H4}$ and $C_{H7}$ appearing in Assumption A\ref{bilinearform_assumption} and Hypotheses H\ref{hypothesis4H} and H\ref{hypothesis8H} depend on $\frac{\delta_{j}}{\delta_{j-1}} $, $\frac{\delta_{j-1}}{\delta_{j}}$ and $\# \mathcal{C}_{E_{j-1}}$, respectively. Therefore, we denote by $C^{j}$, $C_{A3}^{j}$, $C_{H4}^{j}$ and $C_{H7}^{j}$ such constants.
Since as explained in Assumption A\ref{quasi_uniform_mesh}, we assume a bounded variation hypothesis between subsequent levels, then both $\frac{\delta_{j}}{\delta_{j-1}}$ and  $\frac{\delta_{j-1}}{\delta_{j}}$ are bounded. Moreover, if the fine tessellation $\mathcal{T}_{J}$ consisting of triangles is a shape-regular tessellation, then $\# \mathcal{C}_{E_{j-1}}$ is uniformly bounded by $\# \mathcal{C}_{E_{J}}$. Indeed, due to the agglomeration procedure, the cardinality of the set of elements $E_{j} \in \mathcal{T}_{E_{j-1}}$ having a certain node as vertex cannot increase. Hence, all the involved constants are uniformly bounded independently of the level $j$. Consequently, Assumption A\ref{bilinearform_assumption} is satisfied setting $C_{A3} = \max_{j}\{C_{A3}^j\}$ and Assumption A\ref{regularity_approximation_assumption} is satisfied setting $C = \max_{j}\{C^j\}$, $C_{H4} = \max_{j} \{ C_{H4}^{j}\}$ and $C_{H7} = \max_{j} \{ C_{H7}^{j}\}$. Furthermore, Assumption A\ref{smoothing_hypothesis} is satisfied with $C_{A5}$ defined as in \eqref{constant_A4} independently of the level $j$.
To conclude it is sufficient to invoke \cref{V_cycle_convergence,W_cycle_convergence}.

\section{Numerical results} \label{Numerical_experiments}
In this section we describe the implementation of the multigrid method introduced in \Cref{Multigrid_algorithms} and then we present some numerical results to assess the convergence properties of our h-multigrid virtual element algorithm.

\subsection{Implementation details}\label{Implementation}
The algebraic linear system of equations stemming from the virtual element discretization \eqref{PoissonVEM} of the Poisson equation on the finest grid $\mathcal{T}_{J}$ is in the form
\begin{equation}\label{algebraic_problem}
\mathbf{A}_{J} \mathbf{u}_{J} = \mathbf{f}_{J}, 
\end{equation} 
where $\mathbf{u}_{J}\in \mathbb{R}^{\mathcal{N}^{J}_{dof}}$ represents the vector of the degrees of freedom of $u_{J} \in V_{J}$ with respect to the VEM basis, $\mathbf{A} \in \mathbb{R}^{\mathcal{N}^{J}_{dof},\mathcal{N}^{J}_{dof}}$ represents the matrix associated to the operator $A_{J}$ defined in \eqref{operator_A} and $\mathbf{f}_{J} \in \mathbb{R}^{\mathcal{N}^{J}_{dof}}$ is the vector associated to $f_{J} \in V_{j}$ defined as $(f_{J}, v)_{L^{2}(\Omega)} = \sum_{E_J \in \mathcal{T}_J} (f, \Pi^{0}_{0, E_{J}} v)_{L^{2}(E_{J})} \  \forall v \in V_{J}$.

The algebraic counterpart $\mathbf{I}^{j}_{j-1}\in \mathbb{R}^{\mathcal{N}^{j}_{dof},\mathcal{N}^{j-1}_{dof}}$ of the prolongation operator $I_{j-1}^{j}: V_{j-1} \to V_{j}$ is locally defined  $\forall t \in \mathcal{N}(E_{j-1})$ as
\begin{equation*}
{(\mathbf{I}^{j}_{j-1})}_{it} := 
\begin{cases}
\Big ( \sum_{g = 1}^{N_1}  \text{dof}_{\mathbb{P}_1(E_{j-1})}  (\Pi^{0}_{1,E_{j-1}} \varphi_{t}^{E_{j-1}})_{g}     m_{g} (\mathbf{x}_i) \Big ) \ \ \ &i \in \mathcal{N}(\mathcal{T}_{E_{j-1}} \backslash E_{j-1}), \\
\ \ \ 1 \ \ \  & i = t , \ \ \ i \in \mathcal{N}(E_{j-1}),\\
\ \ \ 0 \ \ \  & i \neq t, \ \ \ i \in \mathcal{N}(E_{j-1}),
\end{cases}
\end{equation*}
where  $\text{dof}_{\mathbb{P}_1(E_{j-1})}(\cdot)$ is the operator returning the degrees of freedom with respect to the basis of $\mathbb{P}_1(E_{j-1})$ consisting of the set of scaled monomials $\mathcal{M}_{1}(E_{j-1})$ introduced in \eqref{scaled_monomials}.
The algebraic counterpart of the restriction operator $I_{j-1}^{j}: V_{j} \to V_{j-1}$ is denoted by $\mathbf{I}^{j-1}_{j}\in \mathbb{R}^{\mathcal{N}^{j-1}_{dof},\mathcal{N}^{j}_{dof}}$ and the algebraic counterpart  of the operator $A_{j-1}, \ j = 2, \dots, J$ is denoted by $\mathbf{A}_{j-1} \in \mathbb{R}^{\mathcal{N}^{j-1}_{dof},\mathcal{N}^{j-1}_{dof}}$.

As a smoothing iteration, we have selected the Gauss-Seidel method. The algebraic counterparts of the operators $R_{j} : V_{j} \to V_{j}$  and $R_{j}^{T}$ are the matrix $\mathbf{R}_{j} \in \mathbb{R}^{\mathcal{N}^{j}_{dof},\mathcal{N}^{j}_{dof}}$ and  $\mathbf{R}_{j}^{T} \in \mathbb{R}^{\mathcal{N}^{j}_{dof},\mathcal{N}^{j}_{dof}}$, respectively. We set $\mathbf{R}_{j}^{(l)}$ equals to $\mathbf{R}_{j}$ if $l$ is odd and equals to $\mathbf{R}_{j}^{T} $ if if $l$ is even.

Now, we are ready to introduce the algebraic counterpart of the multigrid method introduced in \Cref{Multigrid_algorithms}. In \cref{Multigrid_algorithm}, we outline the multigrid iteration algorithm for the computation of $\mathbf{u}_{J}$. $\text{MG}_{p}(J, \mathbf{f}_{J}, \mathbf{u}_{k}, \nu)$ represents either one iteration of the non-nested W-cycle ($p = 2$) or one iteration of non-nested V-cycle ($p =  1$).  

\begin{algorithm}
\footnotesize
\caption{Multigrid iteration for the solution of problem \eqref{algebraic_problem}}\label{Multigrid_algorithm}
\begin{algorithmic}
\State{Initialize $\mathbf{u}^{0}$;}
\For{$k = 0, 1, \dots$}\\
\ \ \ \ \ $\mathbf{u}^{k+1} = \text{MG}_{p}(J, \mathbf{f}_{J}, \mathbf{u}^{k}, \nu);$ 
\EndFor
\end{algorithmic}
\end{algorithm}

In particular, \cref{W_cycle} represents the solution obtained after one iteration of either the W-cycle ($p = 2$) or the V-cycle ($p = V$) method with initial guess $\mathbf{x}^{0}$ and $\nu$ Gauss-Seidel iterations of pre-smoothing and post-smoothing. The two-level method is a particular case of \cref{W_cycle} corresponding to $J=2$.

\begin{algorithm}
\footnotesize
\caption{$p$-cycle Multigrid ($p = 1$ or $p = 2$) \ \ \ $\mathbf{y} = \text{MG}_{p}(j,\mathbf{g}, \mathbf{x}^{0}, \nu) $}\label{W_cycle}
\begin{algorithmic}
\State{Set $\mathbf{q}^{0} = 0.$}
\If{ $j = 1$} \\
\ \ \ \ \ $\text{MG}_{p}(1, \mathbf{g}, \mathbf{x}^{0}, \nu) = \mathbf{A}_{1}^{-1} \mathbf{g}.$
\Else
\ \ \ \ \ \State{\underline{\textit{Pre-smoothing:}}}
\For{$ l = 1, \dots, \nu$} \\
\ \ \ \ \ \ \ \ \ \ $\mathbf{x}^{l} = \mathbf{x}^{l-1} + \mathbf{R}_{j}^{(l+\nu)}(\mathbf{g} - \mathbf{A}_{j} \mathbf{x}^{l-1});$
 \EndFor  
 \ \ \ \ \ \State{\underline{\textit{Coarse grid correction:}}} \\
\ \ \ \ \ $\mathbf{r}_{j-1} = \mathbf{I}^{j-1}_{j}(\mathbf{g} - \mathbf{A}_{j} \mathbf{x}^{\nu});$ 
\For{$i = 1, \dots, p$} \\
\ \ \ \ \ \ \ \ \ \ $\mathbf{q}^{i} = \text{MG}_{p}(j-1, \mathbf{r}_{j-1}, \mathbf{{q}}^{i-1}, \nu);$ 
\EndFor

 $\mathbf{y}^{\nu} = \mathbf{x}^{\nu} + \mathbf{I}^{j}_{j-1} \mathbf{q}^{p};$ 
\ \ \ \ \ \State{\underline{\textit{Post-smoothing:}}}
\For{$ l = \nu + 1, \dots, 2 \nu $} \\
\ \ \ \ \ \ \ \ \ \ $\mathbf{y}^{l} = \mathbf{y}^{l-1} + \mathbf{R}_{j}^{(l+\nu)}(\mathbf{g} - \mathbf{A}_{j} \mathbf{y}^{l-1})$;
\EndFor 
\ \ \ \ \ \ \State{ $\text{MG}_{p}(j, \mathbf{g}, \mathbf{x}^{0}, \nu) = \mathbf{y}^{2 \nu}. $ }
\EndIf
\end{algorithmic}
\end{algorithm}

\subsection{Tests}
In this section we present some numerical results to assess the convergence properties of our h-multigrid virtual element algorithm for the solution of the Poisson equation on the unit square $\Omega = (0,1) \times (0,1)$ with $\mu =1 $, $f(x,y) = -2(x(x-1)+y(y-1))$ and homogeneous Dirichlet boundary conditions. We consider both the case in which the bilinear form is inherited and non-inherited, cf. \cref{inherited_non_inherited_bilinear_form}.

We consider the set of agglomerated meshes shown in Figure \ref{Meshes}. The coarsening strategy has been realized through a code developed by the authors. The first row of Figure \ref{Meshes} shows the sequence  of  initial fine grids corresponding to decreasing mesh sizes $\delta_{J}$.  They consist of shape-regular triangle tessellations with $511$ (Figure \ref{a}), $1034$ (Figure \ref{b}), $1939$ (Figure \ref{c}) and $ 3915$ (Figure \ref{d}) elements, respectively. The triangle mesh have been generated using the Triangle library \cite{shewchuk1996triangle}. The remaining rows of Figure \ref{Meshes} show the sequence of agglomerated nested coarsened meshes.

Our aim is to analyse the performance of the two-level, the W-cycle and the  V-cycle h-multigrid schemes based on the virtual element method of order $k=1$. We set a relative tolerance of $10^{-8}$ as a stopping criterion. 


\setcounter{figure}{3}
\begin{figure}
\centering\begin{tabular}{@{ }c@{ }c@{ }c@{ }c@{ }c@{ }}
&\textbf{Set 1} & \textbf{Set 2} & \textbf{Set 3} & \textbf{Set 4} \\
\rowname{}&
\includegraphics[width=0.23\textwidth]{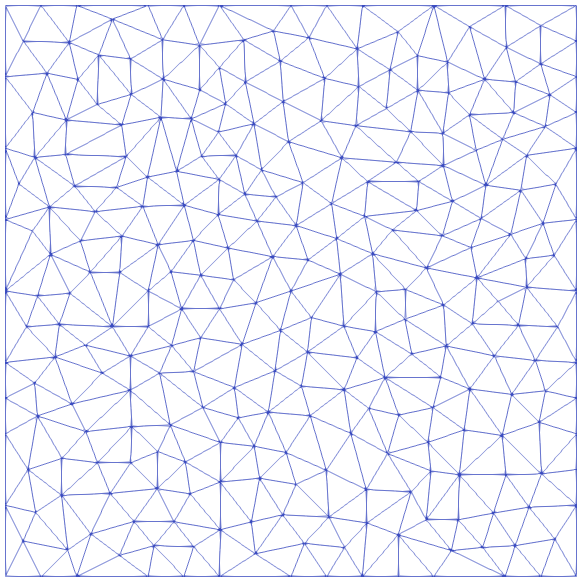}&
\includegraphics[width=0.23\textwidth]{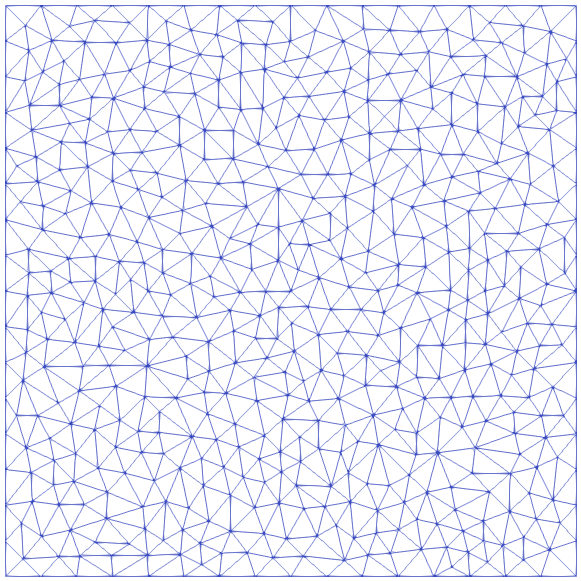}&
\includegraphics[width=0.23\textwidth]{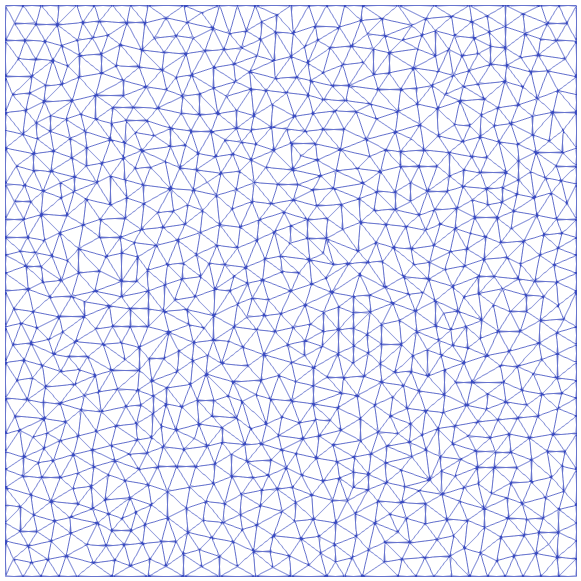}&
\includegraphics[width=0.23\textwidth]{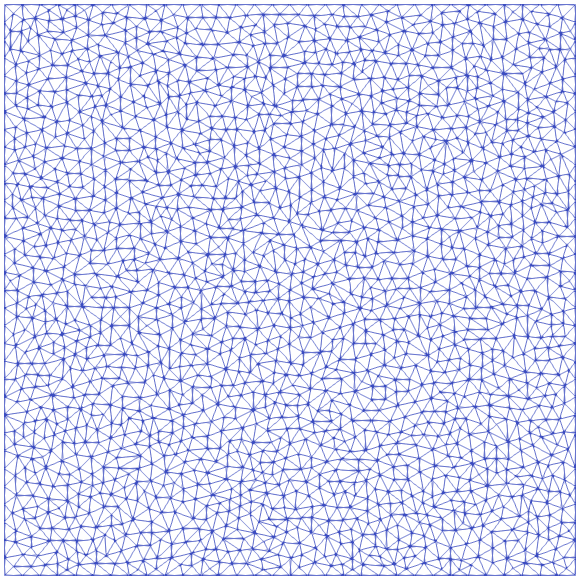}\\
&\mycaption{$\mathcal{T}_{J}, 511$}\label{a} & \mycaption{$\mathcal{T}_{J}, 1034$}\label{b} & \mycaption{$\mathcal{T}_{J}, 1939$}\label{c} & \mycaption{$\mathcal{T}_{J}, 3915$}\label{d}\\
\rowname{}&
\includegraphics[width=0.23\textwidth]{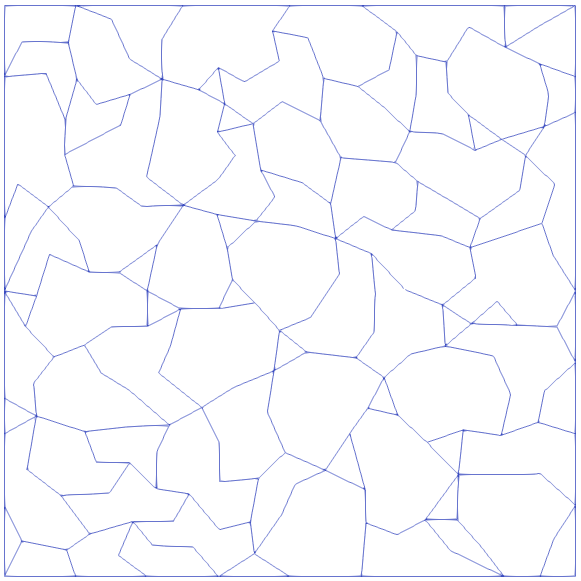}&
\includegraphics[width=0.23\textwidth]{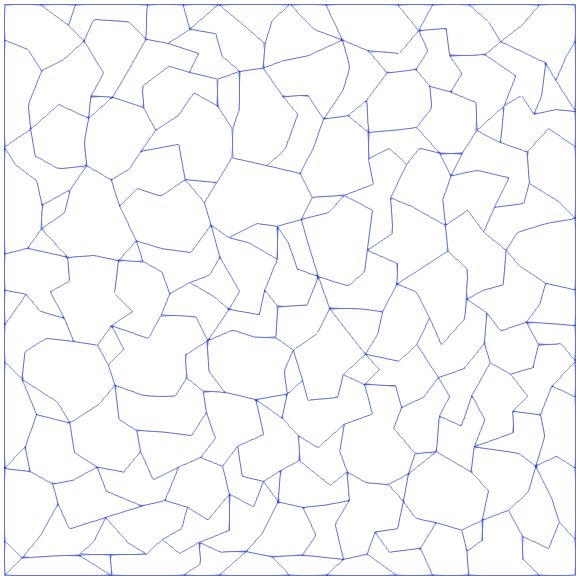}&
\includegraphics[width=0.23\textwidth]{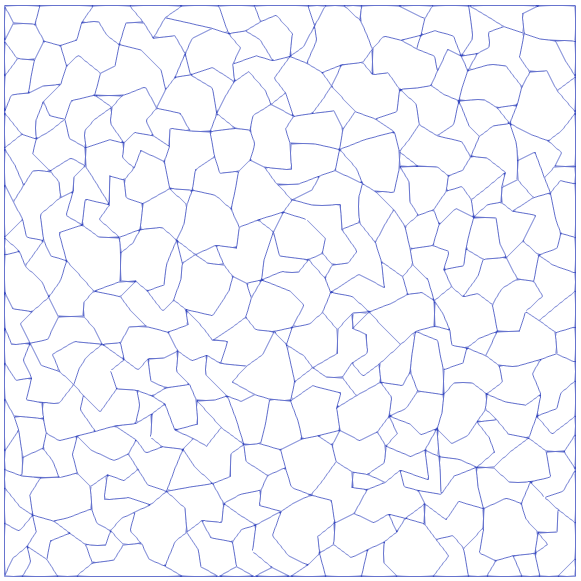}&
\includegraphics[width=0.23\textwidth]{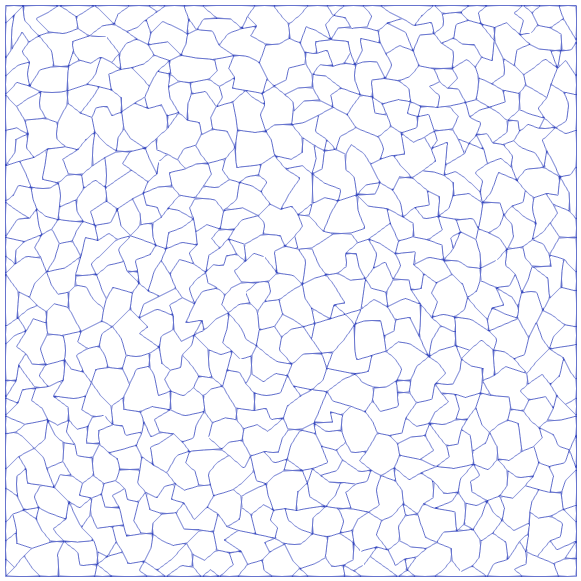}\\
&\mycaption{$\mathcal{T}_{J-1}$} & \mycaption{$\mathcal{T}_{J-1}$} & \mycaption{$\mathcal{T}_{J-1}$} & \mycaption{$\mathcal{T}_{J-1}$}\\
\rowname{}&
\includegraphics[width=0.23\textwidth]{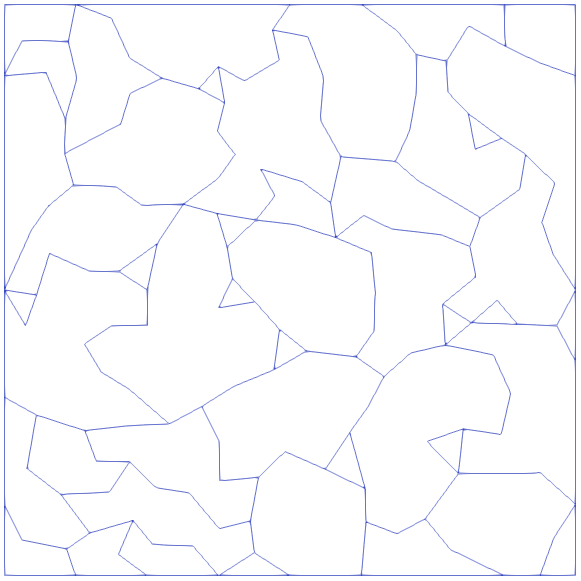}&
\includegraphics[width=0.23\textwidth]{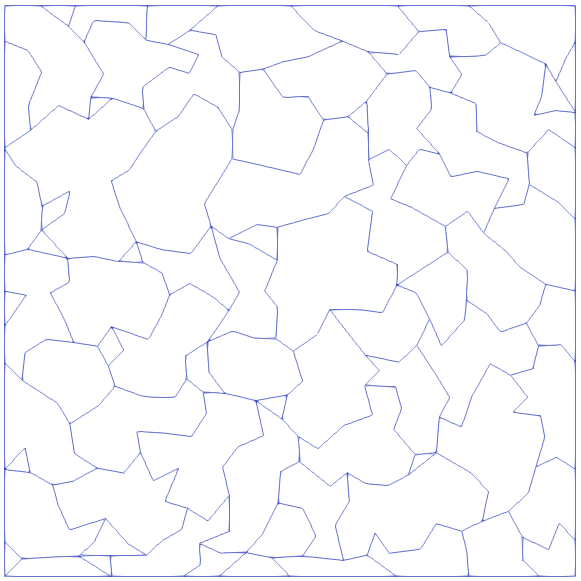}&
\includegraphics[width=0.23\textwidth]{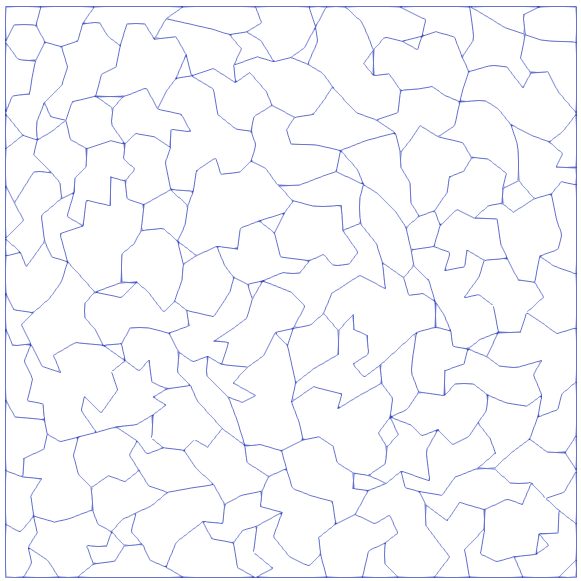}&
\includegraphics[width=0.23\textwidth]{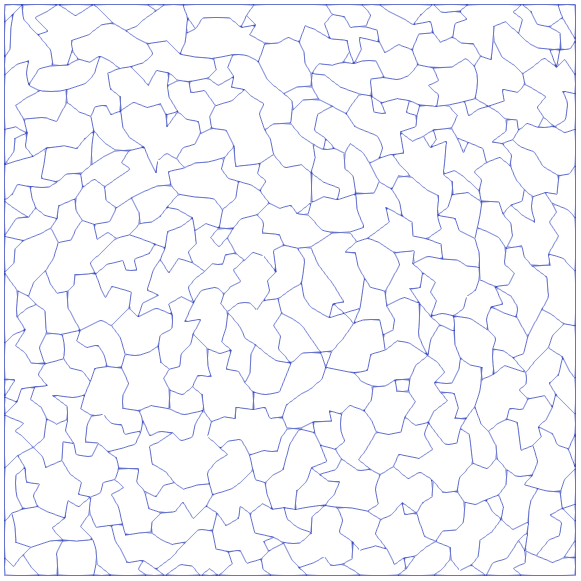}\\
&\mycaption{$\mathcal{T}_{J-2}$} & \mycaption{$\mathcal{T}_{J-2}$} & \mycaption{$\mathcal{T}_{J-2}$} & \mycaption{$\mathcal{T}_{J-2}$}\\
\rowname{}&
\includegraphics[width=0.23\textwidth]{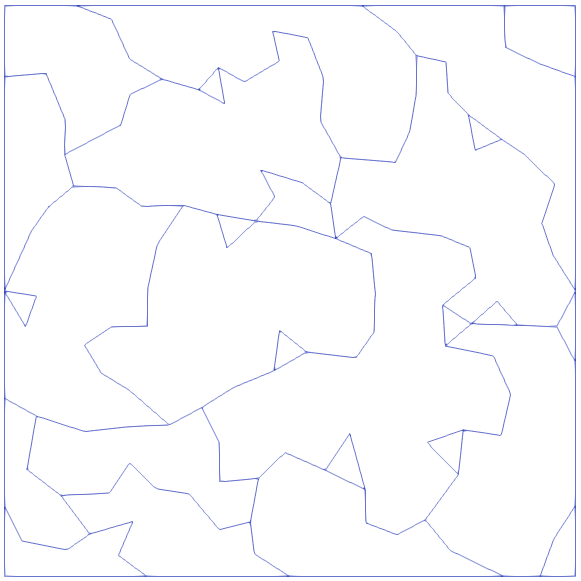}&
\includegraphics[width=0.23\textwidth]{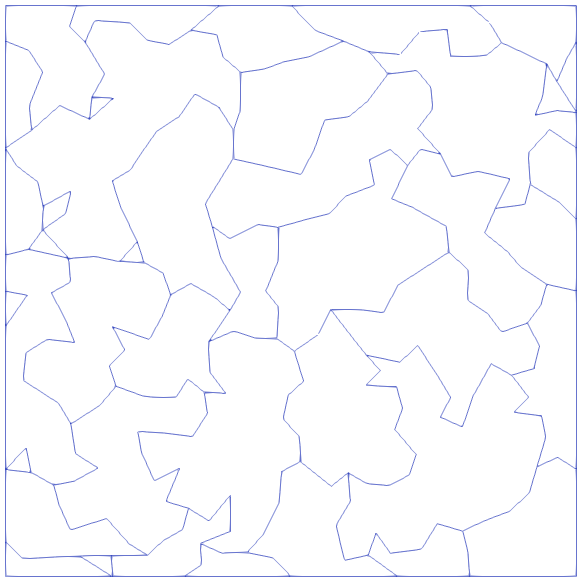}&
\includegraphics[width=0.23\textwidth]{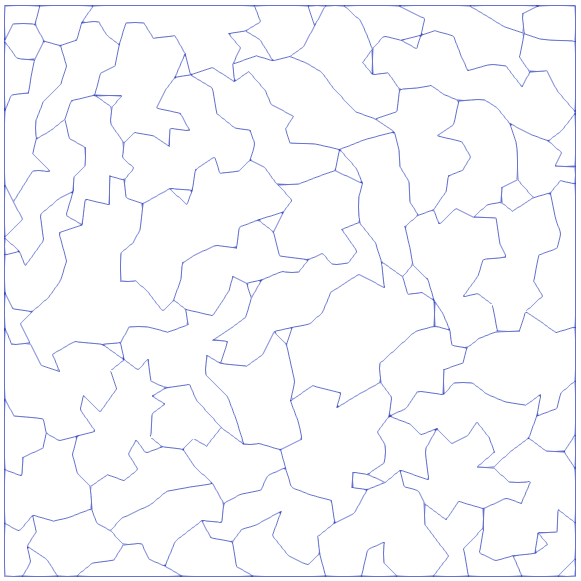}&
\includegraphics[width=0.23\textwidth]{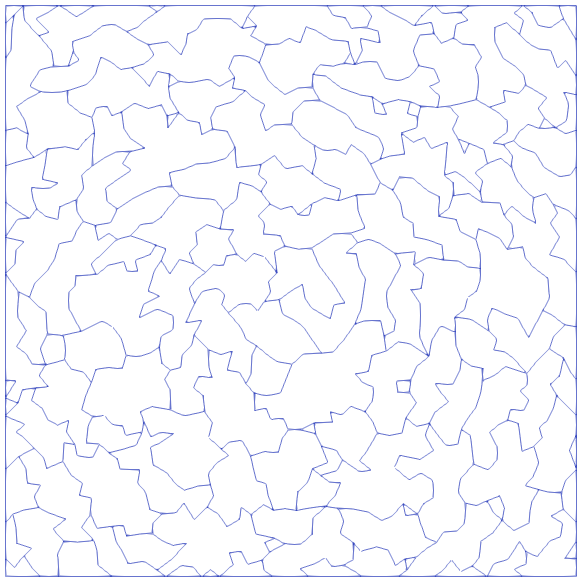}\\
&\mycaption{$\mathcal{T}_{1}$} & \mycaption{$\mathcal{T}_{1}$} & \mycaption{$\mathcal{T}_{1}$} & \mycaption{$\mathcal{T}_{1}$} \\
\end{tabular}
\ContinuedFloat
\caption{Sequences of agglomerated grids for testing the h-multigrid method. The corresponding fine grids $\mathcal{T}_{J}$ consist of 511 (a), 1034 (b), 1939 (c) and 3915 (d) elements, respectively.}%
\label{Meshes}
\end{figure}

In \cref{Num_cycles,Num_cycles2,Num_cycles_noninherited,Num_cycles2_noninherited}, we report the iteration counts (or cycles) needed to reduce the relative residual below the chosen tolerance and the computed convergence factor defined as
$$ \rho := \exp \Big ( \frac{1}{N} \ln \frac{\|\mathbf{r}_{N}\|_2}{\|\mathbf{r}_0\|_2} \Big),$$
where $\mathbf{r}_{N}$ and $\mathbf{r}_0$ are the final and the initial residual vectors, respectively. The number of iterations is presented as function of the number of levels and the number of smoothing steps. The results are shown for the two-level (TL), the W-cycle and the V-cycle multigrid. In particular, in \cref{Num_cycles,Num_cycles2}, we report the results obtained in case the inherited bilinear form is chosen, whereas in \cref{Num_cycles_noninherited,Num_cycles2_noninherited}, we report the results obtained in case the non-inherited bilinear form is selected. From the results of \cref{Num_cycles,Num_cycles2,Num_cycles_noninherited,Num_cycles2_noninherited}, we notice that for a given number of smoothing iterations $\nu$, the  number of iterations needed to reduce the relative residual below the fixed tolerance does not vary significantly with respect to the dimension of the underlying algebraic system, as predicted by  \cref{W_cycle_convergence,V_cycle_convergence}. Moreover, as expected,  the iteration counts decrease for larger values of $\nu$. From \cref{Num_cycles_noninherited,Num_cycles2_noninherited}, we further observe that the assumption on the number of smoothing steps needed to guarantee convergence in the case of non-inherited bilinear form does not seem to play a key role for the considered test case.

In \cref{Num_cycles}, for each set of tessellations, we report also the number of iterations N$_{\text{it}}^{\text{CG}}$ for the Conjugate Gradient (CG) method and the number of iterations N$_{\text{it}}^{\text{PCG}}$ for the Preconditioned Conjugate Gradient (PCG) method accelerated with a Modified Incomplete Cholesky with dual threshold precoditioner. The comparison shows that the proposed method outperforms both the CG and the PCG scheme in terms of number of iterations required to achieve convergence within the prescribed tolerance even for a small value $\nu$ of smoothing steps.

We observe that even if the agglomerated grids obtained by the considered coarsening strategy, in general, do not necessarily strictly satisfy the quasi-uniformity Assumption A\ref{quasi_uniform_mesh}, the numerical results agree with the theoretical expected behaviour. This is probably due to the use of a limited number of agglomeration levels. If a larger number of level $j$ is considered, ad hoc post-processing techniques can improve the quality of the meshes and enforce the satisfaction of Assumption A\ref{quasi_uniform_mesh}. This will be the object of further investigations.

\begin{table}
\begin{center}
\scriptsize
  \begin{tabular}{ l l l l l l l l l  l }
    \hline
    \multirow{2}{*}{} &
      \multirow{2}{*}{TL} &
      \multicolumn{2}{c}{W-cycle} & 
    \multirow{2}{*}{} &
      \multirow{2}{*}{TL} &
      \multicolumn{2}{c}{W-cycle} & \\
          & & 3 level & 4 level   & 
    & & 3 level & 4 level  \\
    \hline
    & \textbf{Set 1} & & &  & \textbf{Set 2} & & & \\
 $\nu = 2$ &  8 (0.092) & 8 (0.092) & 8 (0.092)  & &  9 (0.107) & 9 (0.109) & 9 (0.109)\\
 $\nu = 4$ & 6 (0.035)&  6 (0.036) & 6 (0.036) & &  6 (0.045)& 6 (0.046) & 6 (0.046)\\
 $\nu = 6$ & 5 (0.022) & 5 (0.022) & 5 (0.022)  & & 6 (0.027) & 6 (0.027) & 6 (0.027) \\
 $\nu = 8$ & 5 (0.017) & 5 (0.017) & 5 (0.017)  & &  5 (0.017) & 5 (0.017)  & 5 (0.018) \\
 \hline 
 &  N$^{\text{CG}}_{\text{it}} = 51 $,    & N$^{\text{PCG}}_{\text{it}} = 21$ & &&  N$^{\text{CG}}_{\text{it}} = 71$,  &   N$^{\text{PCG}}_{\text{it}} = 29$ & & \\
 \hline
    \multirow{2}{*}{} &
      \multirow{2}{*}{TL} &
      \multicolumn{2}{c}{W-cycle} & 
    \multirow{2}{*}{} &
      \multirow{2}{*}{TL} &
      \multicolumn{2}{c}{W-cycle} & \\
          & & 3 level & 4 level   & 
    & & 3 level & 4 level  \\
    \hline
    & \textbf{Set 3} & & &  & \textbf{Set 4} & & & \\
 $\nu = 2$ & 8 (0.093) & 8 (0.094) & 8 (0.094) & & 9 (0.105) & 9 (0.105) & 9 (0.105)\\
 $\nu = 4$ & 6 (0.033) & 6 (0.033) & 6 (0.033) & & 6 (0.038)  & 6 (0.038) & 6 (0.038) \\
 $\nu = 6$ & 5 (0.018) & 5 (0.018) & 5 (0.019) & & 5 (0.022) & 5 (0.022) & 5 (0.022) \\
 $\nu = 8$ & 5 (0.013) & 5 (0.013) & 5 (0.013) & & 5 (0.016)  & 5 (0.016) &  5 (0.016) \\
 \hline
 &  N$^{\text{CG}}_{\text{it}} = 102$,    & N$^{\text{PCG}}_{\text{it}} = 40$ &&  & N$^{\text{CG}}_{\text{it}} = 146$, & N$^{\text{PCG}}_{\text{it}} = 58$  & \\
 \hline
  \end{tabular}
  \caption{(\textbf{Inherited case}). Iteration counts and convergence factor (within parentheses) for both the two-level (TL) and the W-cycle algorithms as function of $\nu$ and for the W-cycle scheme as a function of the number of levels. The results are compared with the corresponding iteration counts of the CG/PCG methods. The sequence of agglomerated meshes is shown in Figure \ref{Meshes}.}
  \label{Num_cycles}
  \end{center}
\end{table}

\section{Conclusions}\label{Conclusions}
In this work we have proposed two-level, W-cycle and V-cycle geometric multigrid schemes on agglomeration-based nested polygonal grids and we have theoretically analysed their convergence. In particular, we have focused on the solution of the linear system stemming from a primal Virtual Element discretization of order $k=1$ of the Poisson equations. The novelty of our approach lies in exploiting the flexibility of VEM in dealing with rather general element shapes to generate nested sequences of  tessellations via  a geometric agglomeration procedure. However, the nestedness of the tessellation does not guarantee the nestedness of the virtual element spaces. This crucial aspect has asked for the use of the general BPX framework \cite{bramble1991analysis, duan2007generalized} for non-nested multigrid methods to prove that our multigrid schemes converge uniformly with respect to the mesh size and number of levels. In the case of non-inherited bilinear form the convergence of the  W-cycle scheme is obtained for a sufficiently large number of smoothing steps. Finally, we have validated the effectiveness of our algorithm though numerical experiments.

\begin{table}
\begin{center}
\scriptsize
  \begin{tabular}{ l l l l l l l l l  l }
    \hline
    \multirow{2}{*}{} &
     \multirow{2}{*}{} &
      \multicolumn{2}{c}{V-cycle} & 
    \multirow{2}{*}{} &
     \multirow{2}{*}{} &
      \multicolumn{2}{c}{V-cycle} & \\
          & & 3 level & 4 level   & 
    & & 3 level & 4 level  \\
    \hline
    & \textbf{Set 1} & & &  & \textbf{Set 2} & & & \\
 $\nu = 2$ &  &  9 (0.105) & 9 (0.128)  & & & 10 (0.132)  &  10 (0.150)\\
 $\nu = 4$ &  &  7 (0.050) & 7 (0.066)  & & & \ 7 (0.059)  & \ 8 (0.073)\\
 $\nu = 6$ &  &  6 (0.033) & 6 (0.046) & & & \ 6 (0.034)  & \ 7 (0.048)\\
 $\nu = 8$ &  &  5 (0.025) & 6 (0.034)  & & & \ 5 (0.023)  & \ 6 (0.035) \\
 \hline
    \multirow{2}{*}{} &
      \multirow{2}{*}{TL} &
      \multicolumn{2}{c}{V-cycle} & 
    \multirow{2}{*}{} &
      \multirow{2}{*}{TL} &
      \multicolumn{2}{c}{V-cycle} & \\
          & & 3 level & 4 level   & 
    & & 3 level & 4 level  \\
     \hline
 
    \hline
    & \textbf{Set 3} & &  &  & \textbf{Set 4} & & & \\
 $\nu = 2$ &  & 9 (0.114) & 11 (0.167) &&  & 9 (0.118) & 10 (0.151) \\
 $\nu = 4$ &  & 6 (0.046) & \ 8 (0.077)  &&  & 7 (0.049)  & \ 7 (0.067)  \\
 $\nu = 6$ &  & 6 (0.029) & \ 7 (0.048) &&  & 6 (0.030) & \ 6 (0.042) \\
 $\nu = 8$ &  & 5 (0.020) & \ 6 (0.035)  &&  & 5 (0.021) & \ 6 (0.030) \\
 \hline
  \end{tabular}
  \caption{(\textbf{Inherited case}). Iteration counts and convergence factor (within parentheses) for the  V-cycle scheme as function of $\nu$ and of the number of levels. The sequence of agglomerated meshes is shown in  Figure \ref{Meshes}.}
  \label{Num_cycles2}
  \end{center} 
\end{table}

\begin{table}[H]
\begin{center}
\scriptsize
  \begin{tabular}{ l l l l l l l l l  l }
    \hline
    \multirow{2}{*}{} &
      \multirow{2}{*}{TL} &
      \multicolumn{2}{c}{W-cycle} & 
    \multirow{2}{*}{} &
      \multirow{2}{*}{TL} &
      \multicolumn{2}{c}{W-cycle} & \\
          & & 3 level & 4 level   & 
    & & 3 level & 4 level  \\
    \hline
    & \textbf{Set 1} & & &  & \textbf{Set 2} & & & \\
 $\nu = 2$ & 8 (0.0815) & 8 (0.0817) & 8 (0.0818)  &&  8 (0.0967) & 8 (0.0974) &  8 (0.0975)\\
 $\nu = 4$ & 6 (0.0325) & 6 (0.0327)   & 6 (0.0327)  && 6 (0.0397)  & 6 (0.0399) & 6 (0.0399) \\
 $\nu = 6$ & 5 (0.0206) &    5 (0.0207) & 5 (0.0207) && 5 (0.0207)  & 5 (0.0208) & 5 (0.0208) \\
 $\nu = 8$ & 5 (0.0151) &   5 (0.0151)  & 5 (0.0151) && 5 (0.0123)   & 5 (0.0123) & 5 (0.0123) \\
 \hline 
    \multirow{2}{*}{} &
      \multirow{2}{*}{TL} &
      \multicolumn{2}{c}{W-cycle} & 
    \multirow{2}{*}{} &
      \multirow{2}{*}{TL} &
      \multicolumn{2}{c}{W-cycle} & \\
          & & 3 level & 4 level   & 
    & & 3 level & 4 level  \\
    \hline
    & \textbf{Set 3} & & &  & \textbf{Set 4} & & & \\
 $\nu = 2$ & 8 (0.0885) & 8 (0.0889) & 8 (0.0890)  &&  8 (0.0958) & 8 (0.0958)  & 8 (0.0958)\\
 $\nu = 4$ & 6 (0.0299) & 6 (0.0300) & 6 (0.0300) &&  6 (0.0340) & 6 (0.0341) & 6 (0.0341) \\
 $\nu = 6$ & 5 (0.0160)  & 5 (0.0161)  & 5 (0.0161)  && 5 (0.0189)  & 5 (0.0189)  & 5 (0.0189)\\
 $\nu = 8$ & 5 (0.0112)   & 5 (0.0112) & 5 (0.0112)  &&  5 (0.0126) & 5 (0.0126)   & 5 (0.0126)\\
 \hline
  \end{tabular}
  \caption{(\textbf{Non-inherited case}). Iteration counts and convergence factor (within parenthesis) for both the two-level (TL) and the W-cycle algorithms as function of $\nu$ and for the W-cycle scheme as function of the number of levels. The sequence of agglomerated meshes is shown in Figure \ref{Meshes}.}
  \label{Num_cycles_noninherited}
  \end{center}
\end{table}

\begin{table}[H]
\begin{center}
\scriptsize
  \begin{tabular}{ l l l l l l l l l  l }
    \hline
    \multirow{2}{*}{} &
     \multirow{2}{*}{} &
      \multicolumn{2}{c}{V-cycle} & 
    \multirow{2}{*}{} &
     \multirow{2}{*}{} &
      \multicolumn{2}{c}{V-cycle} & \\
          & & 3 level & 4 level   & 
    & & 3 level & 4 level  \\
    \hline
    & \textbf{Set 1} & & &  & \textbf{Set 2} & & & \\
 $\nu = 2$ & & 8 (0.0924) & 9 (0.1112)    & & &  9 (0.1101) & 9 (0.1259) \\
 $\nu = 4$ & & 6 (0.0421) & 7 (0.0563)      & & &  6 (0.0459) & 7 (0.0600)\\
 $\nu = 6$ & & 6 (0.0273) & 6 (0.0351)    & & & 5 (0.0249) & 6 (0.0375)\\
 $\nu = 8$ & & 5 (0.0197) & 5 (0.0238)   & & & 5 (0.0161) & 6 (0.0267) \\
 \hline
    \multirow{2}{*}{} &
      \multirow{2}{*}{TL} &
      \multicolumn{2}{c}{V-cycle} & 
    \multirow{2}{*}{} &
      \multirow{2}{*}{TL} &
      \multicolumn{2}{c}{V-cycle} & \\
          & & 3 level & 4 level   & 
    & & 3 level & 4 level  \\
     \hline
 
    \hline
    & \textbf{Set 3} & &  &  & \textbf{Set 4} & & \\
 $\nu = 2$ & & 8 (0.0989) & 9 (0.1275)  &&& 8 (0.0972) &  9 (0.1218)  \\
 $\nu = 4$ & & 6 (0.0369) & 7 (0.0619)  &&& 6 (0.0370)  &  7 (0.0565)   \\
 $\nu = 6$ & & 5 (0.0219)& 6 (0.0401) &&& 5 (0.0220) &  6 (0.0352)\\
 $\nu = 8$ & & 5 (0.0162) & 6 (0.0298) &&& 5 (0.0158)  & 6 (0.0256)  \\
 \hline
  \end{tabular}
  \caption{(\textbf{Non-inherited case}). Iteration counts and convergence factor (within parenthesis) for the  V-cycle scheme as function of $\nu$ and of the number of levels. The sequence of agglomerated meshes is shown in Figure \ref{Meshes}.}
  \label{Num_cycles2_noninherited}
  \end{center}
\end{table}

\section*{Acknowledgements}
S. Berrone and M. Busetto acknowledge  that   the   present   research   was   partially   supported   by   MIUR Grant-Dipartimenti di Eccellenza 2018-2022 n. E11G18000350001.  P.F. Antonietti, S.Berrone and M. Verani have been partially funded by MIUR PRIN research grants n. 201744KLJL and n. 20204LN5N5. P.F. Antonientti, S. Berrone, M. Busetto and M. Verani are members of INdAM GNCS.

\vspace{1cm}
\footnotesize
\centering REFERENCES
\vspace{0.5cm}
\bibliography{references.bib}

\begin{thebibliography}{10}
\expandafter\ifx\csname url\endcsname\relax
  \def\url#1{\texttt{#1}}\fi
\expandafter\ifx\csname urlprefix\endcsname\relax\def\urlprefix{URL }\fi
\expandafter\ifx\csname href\endcsname\relax
  \def\href#1#2{#2} \def\path#1{#1}\fi

\bibitem{beirao2013basic}
L.~Beir{\~a}o~da Veiga, F.~Brezzi, A.~Cangiani, G.~Manzini, L.~D. Marini,
  A.~Russo, Basic principles of virtual element methods, Math. Models Methods
  Appl. Sci. 23~(01) (2013) 199--214.

\bibitem{Antonietti2022}
P.~F. Antonietti, L.~Beir{\~a}o~da Veiga, G.~Manzini~\normalfont{(Eds.)}, The
  virtual element method and its applications, SEMA SIMAI Springer Ser.,
  Springer, 2022.

\bibitem{beirao2022}
L.~Beir{\~a}o~da Veiga, N.~Bellomo, F.~Brezzi, L.~D.
  Marini~\normalfont{(Eds.)}, Recent results and perspectives for virtual
  element methods, Special issue in Math. Models Methods Appl. Sci.

\bibitem{berrone2017orthogonal}
S.~Berrone, A.~Borio, Orthogonal polynomials in badly shaped polygonal elements
  for the virtual element method, Finite Elem. Anal. Des. 129 (2017) 14--31.

\bibitem{mascotto2018ill}
L.~Mascotto, Ill-conditioning in the virtual element method: Stabilizations and
  bases, Numer. Methods Partial Differential Equations 34~(4) (2018)
  1258--1281.

\bibitem{calvo2019overlapping}
J.~G. Calvo, An overlapping schwarz method for virtual element discretizations
  in two dimensions, Comput. Math. Appl. 77~(4) (2019) 1163--1177.

\bibitem{calvo2018approximation}
J.~G. Calvo, On the approximation of a virtual coarse space for domain
  decomposition methods in two dimensions, Math. Methods Appl. Sci. 28~(07)
  (2018) 1267--1289.

\bibitem{bertoluzza2017bddc}
S.~Bertoluzza, M.~Pennacchio, D.~Prada, Bddc and feti-dp for the virtual
  element method, Calcolo 54~(4) (2017) 1565--1593.

\bibitem{prada2017feti}
D.~Prada, S.~Bertoluzza, M.~Pennacchio, M.~Livesu, Feti-dp preconditioners for
  the virtual element method on general 2d meshes, in: European Conference on
  Numerical Mathematics and Advanced Applications, Springer, 2017, pp.
  157--164.

\bibitem{dassi2020parallel}
F.~Dassi, S.~Scacchi, Parallel block preconditioners for three-dimensional
  virtual element discretizations of saddle-point problems, Comput. Methods
  Appl. Mech. Engrg. 372 (2020) 113424.

\bibitem{antonietti2018multigrid}
P.~F. Antonietti, L.~Mascotto, M.~Verani, A multigrid algorithm for the
  p-version of the virtual element method, ESAIM Math. Model. Numer. Anal.
  52~(1) (2018) 337--364.

\bibitem{prada2018algebraic}
D.~Prada, M.~Pennacchio, Algebraic multigrid methods for virtual element
  discretizations: A numerical study, arXiv preprint arXiv:1812.02161.

\bibitem{antonietti2017multigrid}
P.~F. Antonietti, P.~Houston, X.~Hu, M.~Sarti, M.~Verani, Multigrid algorithms
  for hp-version interior penalty discontinuous galerkin methods on polygonal
  and polyhedral meshes, Calcolo 54~(4) (2017) 1169--1198.

\bibitem{antonietti2019v}
P.~F. Antonietti, G.~Pennesi, V-cycle multigrid algorithms for discontinuous
  galerkin methods on non-nested polytopic meshes, J. Sci. Comput. 78~(1)
  (2019) 625--652.

\bibitem{bramble1991analysis}
J.~H. Bramble, J.~E. Pasciak, J.~Xu, The analysis of multigrid algorithms with
  nonnested spaces or noninherited quadratic forms, Math. Comp. 56~(193) (1991)
  1--34.

\bibitem{duan2007generalized}
H.-Y. Duan, S.-Q. Gao, R.~Tan, S.~Zhang, A generalized bpx multigrid framework
  covering nonnested v-cycle methods, Math. Comp. 76~(257) (2007) 137--152.

\bibitem{brenner1999convergence}
S.~Brenner, Convergence of nonconforming multigrid methods without full
  elliptic regularity, Mathematics of computation 68~(225) (1999) 25--53.

\bibitem{beirao2016virtual}
L.~Beir{\~a}o~da Veiga, F.~Brezzi, L.~D. Marini, A.~Russo, Virtual element
  method for general second-order elliptic problems on polygonal meshes, Math.
  Models Methods Appl. Sci. 26~(04) (2016) 729--750.

\bibitem{bramble1987new}
J.~H. Bramble, J.~E. Pasciak, New convergence estimates for multigrid
  algorithms, Math. Comp. 49~(180) (1987) 311--329.

\bibitem{chen2018some}
L.~Chen, J.~Huang, Some error analysis on virtual element methods, Calcolo
  55~(1) (2018) 1--23.

\bibitem{ahmad2013equivalent}
B.~Ahmad, A.~Alsaedi, F.~Brezzi, L.~D. Marini, A.~Russo, Equivalent projectors
  for virtual element methods, Comput. Math. Appl. 66~(3) (2013) 376--391.

\bibitem{bramble1992analysis}
J.~H. Bramble, J.~E. Pasciak, The analysis of smoothers for multigrid
  algorithms, Math. Comp. 58~(198) (1992) 467--488.

\bibitem{shewchuk1996triangle}
J.~R. Shewchuk, Triangle: Engineering a 2d quality mesh generator and delaunay
  triangulator, Springer, 1996, pp. 203--222.

\end{thebibliography}

\end{document}